\documentclass[11pt]{amsart}
\usepackage[latin1]{inputenc}
\usepackage{amsmath,amsthm}
\usepackage{amssymb,amsfonts}
\usepackage{hyperref}
\usepackage{tikz}

\usepackage[shortlabels]{enumitem}
\usepackage{bm}
\usepackage{cite}
\usepackage{graphicx}
\usepackage{verbatim}
\usepackage{amsthm}
\usepackage{setspace}
\usepackage{amsmath}
\usepackage{color}

\usepackage[mode=multiuser,status=draft,lang=english]{fixme}
\fxsetup{theme=colorsig}
\FXRegisterAuthor{M}{aM}{MV}

\usepackage[a4paper]{geometry}

\hypersetup{bookmarks,final,colorlinks=true,
linkcolor=[rgb]{0.1, 0.1, 0.44},
citecolor=[rgb]{0.1, 0.1, 0.44}}

\usepackage{stmaryrd}
\DeclareSymbolFont{bbold}{U}{bbold}{m}{n}
\DeclareSymbolFontAlphabet{\mathbbm}{bbold}

\title{Tameness for set theory $II$}
\author{Matteo Viale}
\thanks{The author acknowledge support from INDAM through GNSAGA
and from the project:
\emph{PRIN 2017-2017NWTM8R
Mathematical Logic: models, sets, computability.
}
}

\theoremstyle{plain}
	\newtheorem{Theorem}{Theorem}
	
	\newtheorem{Notation}{Notation}

	\newtheorem{theorem}{Theorem}[section]
	
	\newtheorem{lemma}[theorem]{Lemma}
	\newtheorem{corollary}[theorem]{Corollary}
	\newtheorem{fact}[theorem]{Fact}

	\newtheorem{question}[theorem]{Question}
	\newtheorem{claim}{Claim}
	\newtheorem{subclaim}{Subclaim}
\theoremstyle{definition}
	\newtheorem{Definition}[Theorem]{Definition}

	\newtheorem{definition}[theorem]{Definition}
	\newtheorem{notation}[theorem]{Notation}
	\newtheorem{example}[theorem]{Example}

\theoremstyle{remark}
	\newtheorem{remark}[theorem]{Remark}

\newcommand{\ZFC}{\ensuremath{\mathsf{ZFC}}}
\newcommand{\ZF}{\ensuremath{\mathsf{ZF}}}

\newcommand{\WFE}{\ensuremath{\mathsf{WFE}}}

\DeclareMathOperator{\crit}{crit}

\DeclareMathOperator{\Ult}{Ult}

\DeclareMathOperator{\Coll}{Coll}

\newcommand{\maxUB}{\ensuremath{
{\mathbf{MAX}(\mathsf{UB})}}}
\newcommand{\Pmax}{\ensuremath{\mathbb{P}_{\mathrm{max}}}}
\newcommand{\stUB}{\ensuremath{(*)\text{-}\mathsf{UB}}}
\newcommand{\stA}{\ensuremath{(*)\text{-}\mathcal{A}}}
\newcommand{\NS}{\ensuremath{\mathbf{NS}}} 

\newcommand{\tow}[1]{\mathcal{#1}}

\newcommand{\pow}[1]{\mathcal{P}\left(#1\right)}

\newcommand{\ap}[1]{\langle #1 \rangle}
\newcommand{\bp}[1]{\left\lbrace #1 \right\rbrace}

\newcommand{\Cod}{\ensuremath{\text{{\rm Cod}}}}
\newcommand{\ST}{\ensuremath{\text{{\sf ST}}}}
\newcommand{\UB}{\ensuremath{\text{{\sf UB}}}}
\newcommand{\lUB}{\ensuremath{\text{{\rm l-UB}}}}

\newcommand{\MM}{\ensuremath{\text{{\sf MM}}}} 

\newcommand{\SSP}{\ensuremath{\text{{\sf SSP}}}}

\begin{document}

\begin{abstract}
The paper is the second of two and shows that (assuming large cardinals) set theory is 
a tractable (and we dare to say tame) 
first order theory when formalized in a first order signature with natural predicate symbols for the basic 
definable concepts of second and third order arithmetic, and appealing
to the model-theoretic notions of model completeness and model companionship. 

Specifically we use the general framework linking 
generic absoluteness results to model companionship introduced in the first paper
to show that strong forms of Woodin's axiom $(*)$ entail that
any theory $T$ extending $\ZFC$ by suitable large cardinal axioms has a model companion $T^*$ with respect to certain signatures $\tau$ containing symbols for $\Delta_0$-relations and functions, constant symbols for $\omega$ and 
$\omega_1$, a predicate symbol for the nonstationary ideal on $\omega_1$, symbols for certain lightface definable universally Baire sets.

Moreover  $T^*$ is axiomatized by the $\Pi_2$-sentences $\psi$ for 
$\tau$
such that $T$
proves that 
\[
L(\UB)\models(\Pmax\Vdash\psi^{H_{\omega_2}}),
\]
where $L(\UB)$ denotes the smallest transitive model 
containing the universally Baire sets.

Key to our results is the recent breakthrough of Asper\`o 
and Schindler establishing that a strong form of
Woodin's axiom $(*)$ follows from $\MM^{++}$.

 \end{abstract}

\maketitle


Throughout this paper we assume the reader is familiar with the results and terminology of \cite{VIATAMSTI}.
We will give detailed references of where to find inside \cite{VIATAMSTI} the notations, theorems, and definitions we will use here.

Let us start rightaway stating the main results.

Let $\tau_\ST$ be a signature containing predicate symbols 
$R_\psi$ of arity $m$ for all bounded $\in$-formulae $\psi(x_1,\dots,x_m)$, function symbols
$f_\theta$ of arity $k$ for for all bounded
$\in$-formulae $\theta(y,x_1,\dots,x_k)$, constant symbols 
$\omega$ and $\emptyset$.
$\ZFC_\ST\supseteq\ZFC$ is the $\tau_\ST$-theory
obtained adding axioms which force in each of its $\tau_\ST$-models $\emptyset$ to be interpreted by the empty set,
$\omega$ to be interpreted by the first infinite ordinal,
each $R_\psi$ as the class of $k$-tuples defined by the bounded formula
$\psi(x_1,\dots,x_k)$, each $f_\theta$ as the $l$-ary class function whose graph
is the extension of the bounded formula $\theta(x_1,\dots,x_l,y)$
(whenever $\theta$ defines a functional relation),
see \cite[Notation 2]{VIATAMSTI} for details.

We supplement \cite[Notation 2]{VIATAMSTI} with another piece of notation that will be used throughout the paper.

\begin{Notation}\label{not:keynotation}
\emph{}

\begin{itemize}
\item
$\tau_{\NS_{\omega_1}}$ is the signature $\tau_\ST\cup\bp{\omega_1}\cup\bp{\NS_{\omega_1}}$ with $\omega_1$ 
a constant symbol, $\NS_{\omega_1}$ a unary predicate symbol.

%
\item
$T_{\NS_{\omega_1}}$ is the $\tau_{\NS_{\omega_1}}$-theory
given by $T_\ST$ together with the axioms
\[
\omega_1\text{ is the first uncountable cardinal},
\]
\[
\forall x\;[(x\subseteq\omega_1\text{ is non-stationary})\leftrightarrow\NS_{\omega_1}(x)].
\]

\item
$\ZFC^-_{\NS_{\omega_1}}$ is the $\tau_{\NS_{\omega_1}}$-theory 
\[
\ZFC^-_\ST+T_{\NS_{\omega_1}}.
\]
\item
Accordingly we define $\ZFC_{\NS_{\omega_1}}$.
\end{itemize}
\end{Notation}

Let $\UB$ denote the family of universally Baire sets (see for details \cite[Section 4.2]{VIATAMSTI}), and $L(\UB)$ denote the smallest transitive model of $\ZF$ which contains $\UB$.

\begin{Theorem}\label{Thm:mainthm-1bis}
Let $\mathcal{V}=(V,\in)$ be a model of 
\[
\ZFC+\maxUB+\emph{there is a supercompact cardinal and class many Woodin cardinals},
\]
and $\UB$ denote the family of universally Baire sets in 
$V$.

TFAE
\begin{enumerate}
\item\label{thm:char(*)-modcomp-1}
$(V,\in)$ models $\stUB$;
\item\label{thm:char(*)-modcomp-2}
$\NS_{\omega_1}$ is precipitous\footnote{See \cite[Section 1.6, pag. 41]{STATLARSON}  for a definition of precipitousness and a discussion of its properties. A key observation is that $\NS_{\omega_1}$ being precipitous is independent of $\mathsf{CH}$ (see for example \cite[Thm. 1.6.24]{STATLARSON}), while $\stUB$ entails $2^{\aleph_0}=\aleph_2$  (for example by the results of \cite[Section 6]{HSTLARSON}).

Another key point is that we stick to the formulation of $\Pmax$ as in \cite{HSTLARSON} so to 
be able in its proof to quote verbatim from \cite{HSTLARSON} all the relevant results on $\Pmax$-preconditions we will use.
It is however possible to  develop $\Pmax$ focusing on Woodin's countable tower rather than 
on the precipitousness of $\NS_{\omega_1}$ to develop the notion of $\Pmax$-precondition. Following this approach in
all its scopes, one should be able to reformulate Thm. \ref{Thm:mainthm-1bis}(\ref{thm:char(*)-modcomp-2}) 
omitting the request that
$\NS_{\omega_1}$ is precipitous. We do not explore this venue any further.} and
the $\tau_{\NS_{\omega_1}}\cup\UB$-theory of $V$ has as model companion the
$\tau_{\NS_{\omega_1}}\cup\UB$-theory of $H_{\omega_2}$.
\end{enumerate}
\end{Theorem}

(\ref{thm:char(*)-modcomp-1}) implies (\ref{thm:char(*)-modcomp-2}) does not 
need the supercompact cardinal.

We give rightaway the definitions of $\maxUB$ and $\stUB$.

\begin{Definition}\label{Keyprop:maxUB}
$\maxUB$: 
There are class many Woodin cardinals in $V$, and  for all 
$G$ $V$-generic for some forcing notion $P\in V$:
\begin{enumerate}
\item\label{Keyprop:maxUB-1}
Any subset of $(2^\omega)^{V[G]}$ definable in $(H_{\omega_1}^{V[G]}\cup\mathsf{UB}^{V[G]},\in)$ 
is universally Baire in $V[G]$.
\item\label{Keyprop:maxUB-2}
Let $H$ be $V[G]$-generic for some forcing notion $Q\in V[G]$. 
Then\footnote{Elementarity is witnessed via the map defined by $A\mapsto A^{V[G][H]}$ for
$A\in \UB^{V[G]}$ and the identity on
$H_{\omega_1}^{V[G]}$ (See \cite[Notation 4.6]{VIATAMSTI} 
for the definition of $A^{V[G][H]}$).}:
\[
(H_{\omega_1}^{V[G]}\cup\UB^{V[G]},\in) \prec 
(H_{\omega_1}^{V[G][H]}\cup\UB^{V[G][H]},\in).
\] 
\end{enumerate}
\end{Definition}

We will comment more on $\maxUB$ in Section \ref{sec:Homega2}; 
for now we observe that $\maxUB$ is a form of sharp for the family of universally Baire sets 
which holds if $V$ has class many Woodin cardinals and 
is a generic extension obtained by collapsing a supercompact cardinal to become countable 
($\maxUB$ is a weakening of the conclusion of
\cite[Thm 3.4.17]{STATLARSON}). Moreover if $\maxUB$ holds in $V$, it
remains true in all further set forcing extensions of $V$. It is open whether $\maxUB$ is a direct consequence of 
suitable large cardinal axioms.

We now turn to the definition of $\stUB$, a natural maximal strengthening of Woodin's axiom $(*)$.
Key to all results of this paper is an analysis of the properties of
generic extensions by $\Pmax$ of $L(\UB)$.
In this analysis $\maxUB$ is used to argue (among other things)
that all sets of reals definable in $L(\UB)$ are universally Baire, so that most of the results established 
in \cite{HSTLARSON}
on the properties of $\Pmax$ for $L(\mathbb{R})$ can be
also asserted for $L(\UB)$.
We will use various forms of Woodin's axiom $(*)$ each stating that $\NS_{\omega_1}$ is saturated together with the existence of $\Pmax$-filters meeting certain families of dense subsets 
of $\Pmax$ definable in $L(\UB)$. 
However in this paper we will not define the $\Pmax$-forcing. 
The reason is that in the proof of all our results, we will use equivalent characterizations of the proper forms of $(*)$ which do not mention at all $\Pmax$. 
We will give at the proper stage the relevant definitions. Meanwhile we assume the reader is familiar with $\Pmax$ or can accept as a blackbox its existence as a certain forcing notion; our reference on this topic 
is \cite{HSTLARSON}.

\begin{Definition}
Let $\mathcal{A}$ be a family of dense subsets of $\Pmax$.
\begin{itemize}
\item
$\stA$ holds if $\NS_{\omega_1}$ is saturated\footnote{See \cite[Section 1.6, pag. 39]{STATLARSON} for a discussion of saturated ideals on $\omega_1$.} and 
there exists a filter $G$ on $\Pmax$ meeting all the dense sets in 
$\mathcal{A}$.
\item
$\stUB$ holds
if $\NS_{\omega_1}$ is saturated and there exists an $L(\UB)$-generic filter $G$ on $\Pmax$. 
\end{itemize}
\end{Definition}

Woodin's definition of $(*)$ \cite[Def. 7.5]{HSTLARSON}
is equivalent to $\stA+$\emph{there are class many Woodin cardinals} 
for $\mathcal{A}$ the family of 
dense subsets of $\Pmax$ existing in $L(\mathbb{R})$.

A key role in all proofs is played by the following generic absoluteness result:

\begin{Theorem}\label{thm:PI1invomega2}
Assume\footnote{We follow the convention introduced in \cite[Notation 2.1]{VIATAMSTI} to define $(V,\tau_{\NS_{\omega_1}}^V)$.} 
$(V,\tau_{\NS_{\omega_1}}^V)$ models $\ZFC_{\NS_{\omega_1}}+$\emph{ there are class many Woodin cardinals}. 
Then the $\Pi_1$-theory of $V$ for the language $\tau_{\NS_{\omega_1}}\cup\UB$ is invariant under set sized forcings.
\end{Theorem}

An objection to Thm. \ref{Thm:mainthm-1bis} is that it subsumes the Platonist standpoint that there exists a definite universe of sets.
At the prize of introducing another bit of notation, we can prove a version of Thm. \ref{Thm:mainthm-1bis} which makes perfect sense also
to a formalist.

\begin{Notation}
\emph{}

\begin{itemize}
\item
$\sigma_\ST$ is the signature containing a 
predicate symbol 
$S_\phi$ of arity $n$ for any $\tau_\ST$-formula $\phi$ with $n$-many free variables.
\item
$\sigma_{\omega,\NS_{\omega_1}}$ is
the signature $\tau_\ST\cup\sigma_\ST$.
%
\item
$T_\lUB$ is the 
$\sigma_{\omega,\NS_{\omega_1}}$-theory 
given by the axioms
\[
\forall x_1\dots x_n\,[S_\psi(x_1,\dots,x_n)\leftrightarrow 
(\bigwedge_{i=1}^n x_i\subseteq \omega^{<\omega}\wedge \psi^{L(\UB)}(x_1,\dots,x_n))]
\]
as $\psi$ ranges over the $\tau_\ST$-formulae.
\item
$\ZFC^{*-}_{\lUB}$ is the $\sigma_{\omega}$-theory 
\[
\ZFC^-_{\ST}\cup T_\lUB;
\]
\item
$\ZFC^{*-}_{\lUB,\NS_{\omega_1}}$ is the $\sigma_{\omega,\NS_{\omega_1}}$-theory 
\[
\ZFC^-_{\NS_{\omega_1}}\cup T_\lUB;
\]
\item
Accordingly we define $\ZFC^{*}_{\lUB}$, $\ZFC^*_{\lUB,\NS_{\omega_1}}$.
\end{itemize}
\end{Notation}

A key observation is that $\ZFC^-_\ST$,
$\ZFC^-_{\NS_{\omega_1}}$, $\ZFC^{*-}_{\lUB}$, $\ZFC^{*-}_{\lUB,\NS_{\omega_1}}$ are all 
\emph{definable} extension of $\ZFC$; more precisely any $\in$-structure $(M,E)$
of $\ZFC^-$ 
admits a unique extension to a 
$\tau$-structure satisfying the extra axioms outlined in the above items for $\tau$ among
the signatures written above 
(for $\tau_\ST\cup\bp{\omega_1,\NS_{\omega_1}}$
the $\in$-model must satisfy
the sentence stating the existence of a smallest uncountable cardinal). The same considerations apply to 
$\ZFC_\ST$,
$\ZFC_{\NS_{\omega_1}}$, $\ZFC^{*}_{\lUB}$, $\ZFC^{*}_{\lUB,\NS_{\omega_1}}$.

%

\begin{Theorem} \label{Thm:mainthm-1}
Let $T$ be any $\sigma_{\omega,\NS_{\omega_1}}$-theory extending 
\[
\ZFC^*_{\lUB,\NS_{\omega_1}}+\maxUB+\text{ there is a supercompact cardinal and  class many Woodin cardinals}.
\]
Then $T$ has a model companion $T^*$. 

Moreover TFAE for any for any $\Pi_2$-sentence $\psi$ for 
$\sigma_{\omega,\NS_{\omega_1}}$:
\begin{enumerate}[(A)]
\item \label{Thm:mainthm-1A}
$T^*\vdash \psi$;
\item \label{Thm:mainthm-1B}
For any complete theory 
\[
S\supseteq T,
\] 
$S_\forall\cup\bp{\psi}$ is consistent;
\item \label{Thm:mainthm-1C}
$T$ proves\footnote{$\dot{H}_{\omega_2}$ denotes a canonical $P$-name for $H_{\omega_2}$ as computed in generic extension by $P$.} 
\[
\exists P \,(P\text{ is a partial order }\wedge 
\Vdash_P\psi^{\dot{H}_{\omega_2}});
\]
\item \label{Thm:mainthm-1D}
$T$ proves
\[
L(\UB)\models[\Pmax\Vdash \psi^{\dot{H}_{\omega_2}}];
\]
\item \label{Thm:mainthm-1E}
\[
T_\forall+\ZFC^*_{\lUB,\NS_{\omega_1}}+\maxUB+\stUB\vdash\psi^{H_{\omega_2}}.
\]
\end{enumerate}

\end{Theorem}

Crucial to the proof of Theorems~\ref{Thm:mainthm-1} and \ref{Thm:mainthm-1bis}
is the recent breakthrough of 
Asper\'o and Schindler \cite{ASPSCH(*)} establishing that $(*)$-$\UB$
follows from $\MM^{++}$.

The paper is organized as follows:

\begin{itemize}
\item Section \ref{sec:Homega1} shows that for many 
natural signatures $\sigma_\mathcal{A}=\tau_\ST\cup\mathcal{A}$
given by certain 
families $\mathcal{A}$ of universally Baire sets, the 
the $\sigma_\mathcal{A}$-theory of $H_{\aleph_1}$ is the model companion of the $\sigma_\mathcal{A}$-theory of $V$.
These results are preliminary to the proofs of Thm.
\ref{Thm:mainthm-1}, 
\ref{Thm:mainthm-1bis}.
\item Section \ref{sec:Homega2} proves Theorems 
\ref{Thm:mainthm-1bis}, \ref{thm:PI1invomega2}, \ref{Thm:mainthm-1}. 
\end{itemize}


Our objective is to make this paper accessible to the widest possible audience (which is however limited to scholars with a strong background in forcing and large cardinals), 
this has been done at the expenses of its brevity.
We tried as much as possible to make the reading of Section 
\ref{sec:Homega1}
accessible also to readers unfamiliar with the stationary tower forcing and with
$\Pmax$. We also tried to formulate the main results of  in such a way that the
use of stationary tower forcing is confined to their proofs, and does not hamper the 
comprehension of the key ideas. This is unfortunately not possible
for many of the results in Section \ref{sec:Homega2}, where a great familiarity with the content of 
\cite{STATLARSON,HSTLARSON} is needed and assumed.
We also decided to give (overly?)
detailed arguments for all non-trivial proofs. 
Almost all proofs in 
Section \ref{sec:Homega2} 
employ the key results on the properties of $\Pmax$-forcing presented in \cite{HSTLARSON}. 
The unique proof containing mathematical ideas not at 
all present in \cite{STATLARSON,HSTLARSON} is that of
Thm \ref{thm:char(*)}, in this case we are inspired by \cite[Lemma 3.2]{ASPSCH(*)}.

\textbf{Acknowledgements:}
This research has been completed while visiting the \'Equipe de Logique Math\'ematique of the IMJ in 
Paris 7 in the fall semester of 2019.
The author thanks Boban Veli\v{c}kovi\'c, David Asper\'o, and Giorgio Venturi 
for the many fruitful discussions held on the topics
of the present paper.

\section{Model companionship versus generic absoluteness for the theory of $H_{\aleph_1}$}\label{sec:Homega1}

\subsection{Model companionship for the theory of $H_{\aleph_1}$}

\begin{definition}
Let $(V,\in)$ be a model of $\ZFC$ and $N\subseteq V$ be a transitive class (or set) which is a model of 
$\ZFC^-$.
$\mathcal{A}\subseteq\UB^V$ is $N$-closed
if whenever
$B\subseteq (2^\omega)^k$ is such that for some 
$\in$-formula $\phi(x_0,\dots,x_n)$
\[
B=\bp{(r_0,\dots,r_{k-1})\in (2^\omega)^k:\, (N,\in,A_0,\dots,A_{n-k})\models\phi(r_0,\dots,r_{k-1},A_0,\dots,A_{n-k})}
\]
with $A_0,\dots,A_{n-k}\in\mathcal{A}$, we have that $B\in\mathcal{A}$.

\end{definition}

\begin{example}
Given a model $(V,\in)$ of $\ZFC+$\emph{there are class many Woodin cardinals},
simple examples of $H_{\omega_1}$-closed families (which we will use) are:
\begin{enumerate}
\item
$\UB^V$, i.e. the family of \emph{all} universally Baire sets of $V$.
\item
$\lUB^V$, i.e. the subsets of $(2^\omega)^k$ (as $k$ varies in the natural) which are the extension
of some $\in$-formula relativized to $L(\UB)$.
\item The family $\UB\cap X$ for some $X\prec V_\theta$ with $\theta$ inaccessible.
\end{enumerate}
\end{example}

\begin{theorem}\label{thm:modcompanHomega1}
Let $(V,\in)$ be a model of 
$\ZFC$, 
and assume
$\mathcal{A}$ is $H_{\omega_1}$-closed.

Let $\tau_{\mathcal{A}}=\tau_\ST\cup\mathcal{A}$.
Then the $\tau_{\mathcal{A}}$-theory of $H_{\omega_1}$ 
is model complete and is the
model companion of the $\tau_{\mathcal{A}}$-theory of $V$.
\end{theorem}

\begin{proof}
Let $T$ be the $\tau_{\mathcal{A}}$-theory of $V$ and $T^*$ be the 
$\tau_{\mathcal{A}}$-theory of $H_{\omega_1}$.

By Levy's absoluteness Lemma \cite[Lemma 4.1]{VIATAMSTI}
\[
(H_{\omega_1},\tau_{\mathcal{A}}^V)\prec_1(V,\tau_{\mathcal{A}}^V),
\]
hence the two structures share the same $\Pi_1$-theory.
Therefore (by the standard characterization of model companionship --- \cite[Thm. 3.18]{VIATAMSTI}) 
it suffices to prove that $T^*$ is model complete.

By Robinson's test \cite[Lemma 3.14(c)]{VIATAMSTI}, 
it suffices to show that any existential $\tau_{\mathcal{A}}$-formula is 
$T^*$-equivalent to a universal $\tau_{\mathcal{A}}$-formula.

Let $A_1,\dots,A_k$ be the predicates in $\mathcal{A}$ appearing in $\phi$.

Let\footnote{See \cite[Def. 2.2]{VIATAMSTI} for a definition of $\Cod_\omega$.} 
\[
B=\bp{(r_1,\dots,r_n)\in (2^\omega)^n: 
(H_{\omega_1},\tau_\ST^V,A_1,\dots,A_k)\models \phi(\Cod_\omega(r_1),\dots,\Cod_\omega(r_n))}.
\]
Then $B$ belongs to $\mathcal{A}$, since $\mathcal{A}$ is  $H_{\omega_1}$-closed. 
Now for any $a_1,\dots,a_n\in H_{\omega_1}$:
\[
(H_{\omega_1},\tau_{\mathcal{A}}^V)\models \phi(a_1,\dots,a_n)
\]
\center{ if and only if }
\[
\forall r_1\dots r_n \bigwedge_{i=1}^n\Cod_\omega(r_i)=a_i\rightarrow B(r_1,\dots,r_n).
\]



This yields that
\[
T^*\vdash 
\forall x_1,\dots,x_n\,(\phi(x_1,\dots,x_n)\leftrightarrow\theta_\phi(x_1,\dots,x_n)).
\]
where $\theta_\phi(x_1,\dots,x_n)$ is the $\Pi_1$-formula in the predicate $B\in\mathcal{A}$
\[
\forall y_1,\dots,y_n\,[(\bigwedge_{i=1}^n x_i=\Cod_\omega(y_i))\rightarrow B(y_1,\dots,y_n)].
\]
\end{proof}

We leave to the reader to check that  the above proof yields the following:
\begin{corollary}
Let $T\supseteq \ZFC_{\lUB}$ 
be a $\sigma_\omega$-theory.
Then $T$ has as model companion
the $\Pi_2$-sentences $\psi$ for $\sigma_\omega$ such that
\[
T\vdash\psi^{H_{\omega_1}}.
\]
\end{corollary}

Finally we will need the following observation:

\begin{fact}\label{fac:keyfacHomega1clos}
Assume $(V,\in)$ models that there are class many Woodin cardinals and
$\mathcal{A}\subseteq \UB$ is $H_{\omega_1}$-closed in $V$.
Let $G$ be $V$-generic for some forcing $P\in V$.

Then $\bp{A^{V[G]}:A\in \mathcal{A}}$ is $H_{\omega_1}$-closed in $V[G]$.
\end{fact}
\begin{proof}
The assumptions grant that
\[
(H_{\omega_1}^{V},\tau_\ST^{V},A:A\in\mathcal{A})\prec
(H_{\omega_1}^{V[G]},\tau_\ST^{V[G]},A:A^{V[G]}\in\mathcal{A})
\]
(by \cite[Thm. 4.7]{VIATAMSTI}).
The very definition of being $H_{\omega_1}$-closed gives that the same sentences 
holding in $(H_{\omega_1}^{V},\tau_\ST^{V},A:A\in\mathcal{A})$
granting in $V$ that $\mathcal{A}$ is $H_{\omega_1}$-closed, also grant that
$\bp{A^{V[G]}:A\in \mathcal{A}}$ is $H_{\omega_1}$-closed in $V[G]$.
\end{proof}

\subsection{$\maxUB$}

From now on we will need in several occasions that $\maxUB$ holds in $V$ (recall Def. \ref{Keyprop:maxUB}).
We will always explicitly state where this assumption is used, hence if a statement
 does not mention it in the hypothesis, 
the assumption is not needed for its thesis. 

 We will use both properties of 
$\maxUB$ crucially: (\ref{Keyprop:maxUB-1}) is used in the proof of Lemma \ref{lem:UBcorr};
(\ref{Keyprop:maxUB-2}) in the proof of Fact \ref{fac:densityUBcorrect}.
Similarly they are essentially used in Remark \ref{rmk:maxUBec}.
Specifically we will need $\maxUB$ to prove that 
certain subsets of $H_{\omega_1}$
simply definable using an existential formula quantifying over $\UB$ 
are coded by a universally Baire set, 
and that this coding is absolute between
generic extensions, i.e. if 
\[
\bp{x\in H_{\omega_1}^V: (H_{\omega_1}\cup\UB,\tau_\ST^V)\models \phi(x)}
\] 
is coded by $A\in \UB^V$,
\[
\bp{x\in H_{\omega_1}^{V[G]}: (H_{\omega_1}^{V[G]}\cup\UB^{V[G]},\tau_\ST^{V[G]}))\models \phi(x)}
\] 
is coded by $A^{V[G]}\in \UB^{V[G]}$ for $\phi$ some $\tau_\ST$-formula\footnote{Note that
the structures $(H_{\omega_1}\cup\UB,\in)$, $(H_{\omega_1}\cup\UB,\tau_\ST^V)$ have the same algebra of definable sets, hence we will use one or the other as we deem most convenient,
since any set definable by some formula in one of these structures is also defined by a possibly different formula in the other. The formulation of $\maxUB$ is unaffacted if we choose any of the two structures as the one for which we predicate it.}.

It is useful to outline what is the different expressive power of the structures
 $(H_{\omega_1},\tau_{\ST}^V,A: A\in\UB^V)$ and 
 $(H_{\omega_1}\cup\UB^V,\tau_{\ST}^V)$.
 The latter can be seen as a second order extension of $H_{\omega_1}$, where
we also allow formulae to quantify over the family of universally Baire subsets of $2^{\omega}$;
in the former quantifiers only range over elements of $H_{\omega_1}$, but we can 
use the universally Baire subsets of $H_{\omega_1}$ as parameters.
This is in exact analogy between the comprehension scheme for the Morse-Kelley axiomatization of set theory
(where formulae with quantifiers ranging over classes are allowed) and the 
comprehension scheme for G\"odel-Bernays axiomatization of set theory
(where just formulae using classes as parameters and quantifiers ranging only over sets are allowed).
To appreciate the difference between the two set-up, 
note that that the axiom of determinacy for universally Baire sets
is expressible in
\[
(H_{\omega_1}\cup\UB,\tau_{\ST}^V)
\]
by the
$\tau_\ST$-sentence
\begin{quote}
\emph{For all $A\subseteq 2^\omega$ there is a winning strategy for one of the players in the game with payoff 
$A$},
\end{quote}
while in 
\[
(H_{\omega_1},\tau_{\ST}^V,A:A\in\UB^V)
\]
it is expressed by the axiom schema of $\Sigma_1$-sentences for $\tau_\ST\cup\bp{A}$
\begin{quote}
\emph{There is a winning strategy for some player in the game with payoff 
$A$}
\end{quote}
as $A$ ranges over the universally Baire sets.


We will crucially use the stronger expressive power of the structure $(H_{\omega_1}\cup\UB,\in)$
to define 
certain universally Baire sets as the extension in $(H_{\omega_1}\cup\UB,\tau_\ST^V)$ of lightface 
$\Sigma_1$-properties (according to the Levy hierarchy); properties which 
require an existential quantifier ranging over all universally Baire sets.



\section{Model companionship versus generic absoluteness for the theory of $H_{\aleph_2}$}\label{sec:Homega2}

This section is devoted to the proofs of Theorems~\ref{Thm:mainthm-1} and~\ref{Thm:mainthm-1bis}.
Along the way we will also prove (and use) Theorem \ref{thm:PI1invomega2}.


Let us give a general outline of these proofs before getting into details. From now on we assume 
the reader 
is familiar with the basic theory of $\Pmax$ as exposed in \cite{HSTLARSON}.

\begin{notation}
For a given family of universally Baire sets $\mathcal{A}$, 
$\sigma_\mathcal{A}$ is the signature $\tau_{\ST}\cup\mathcal{A}$,
$\sigma_{\mathcal{A},\NS_{\omega_1}}$ is the signature $\tau_{\NS_{\omega_1}}\cup\mathcal{A}$.
\end{notation}
The key point is to prove (just on the basis that $(V,\in)\models\maxUB+\stUB$) 
the model completeness of the
$\sigma_{\UB,\NS_{\omega_1}}$-theory of
$H_{\omega_2}$ assuming $\stUB$. To do so we use Robinson's test and we show the following:

\begin{quote}
Assuming $\maxUB$ there is a \emph{special} universally Baire set
$\bar{D}_{\UB,\NS_{\omega_1}}$ defined by an $\in$-formula \emph{(in no parameters)} 
relativized to $L(\UB)$
coding a family of $\Pmax$-preconditions with the following fundemental property:

\emph{
For any 
$\sigma_{\UB,\NS_{\omega_1}}$-formula 
$\psi(x_1,\dots,x_n)$ mentioning the universally Baire predicates $B_1,\dots,B_k$,
there is an algorithmic procedure which finds a \emph{universal} $\sigma_{\UB,\NS_{\omega_1}}$-formula 
$\theta_\psi(x_1,\dots,x_n)$ mentioning just the universally Baire predicates 
$B_1,\dots,B_k,\bar{D}_{\UB,\NS_{\omega_1}}$ such that}
\[
(H_{\omega_2}^{L(\UB)[G]},\sigma_{\bp{B_1,\dots,B_k,\bar{D}_{\UB,\NS_{\omega_1}}},\NS_{\omega_1}}^{L(\UB)[G]})
\models
\forall\vec{x}\,(\psi(x_1,\dots,x_n)\leftrightarrow\theta_\psi(x_1,\dots,x_n))
\]
\emph{whenever $G$ is $L(\UB)$-generic for $\Pmax$.}
\end{quote}
Moreover the definition of $\bar{D}_{\UB,\NS_{\omega_1}}$ and the computation
of $\theta_\psi(x_1,\dots,x_n)$ from $\psi(x_1,\dots,x_n)$ are just based on the assumption that
$(V,\in)$ is a model of $\maxUB$, hence can be replicated mutatis-mutandis in any model of 
$\ZFC+\maxUB$. 
We will need that $(V,\in)$ is a model of $\maxUB+\stUB$ just to argue that
in $V$ there is an $L(\UB)$-generic filter $G$ for $\Pmax$ such that\footnote{It is this part of our argument
where the result of Asper\`o and Schindler establishing the consistency of $\stUB$ relative to a supercompact is used in an essential way. We will address again the role of Asper\`o and Schindler's result in all our proofs in  some closing remarks.}
$H_{\omega_2}^{L(\UB)[G]}=H_{\omega_2}^V$. 
%
Since in all our arguments we will only use that $(V,\in)$ is a model of $\maxUB$ and (in some of them 
also of $\stUB$),
we will be in the position to conclude easily for the truth of Theorem~\ref{Thm:mainthm-1} and~\ref{Thm:mainthm-1bis}.

We condense the above information in the following:

\begin{theorem}\label{thm:keythmmodcompanHomega2}
There is an $\in$-formula $\phi_{\UB,\NS_{\omega_1}}(x)$ in one free variable
such that:
\begin{enumerate}
\item
$\ZFC^*_{\lUB}+\maxUB$
proves that 
\emph{$S_{\phi_{\UB,\NS_{\omega_1}}}$ is universally Baire}.
\item
Given predicate symbols $B_1,\dots,B_k$, consider the theory $T_{B_1,\dots,B_k}$ in signature 
$\sigma_{\omega}\cup\bp{B_1,\dots,B_k}$
extending $\ZFC^*_{\lUB}+\maxUB$ by the axioms:
\[
B_j\text{ is universally Baire}
\]
for all predicate symbols $B_1,\dots,B_k$.

There is a recursive procedure assigning to any \emph{existential}
formula $\phi(x_1,\dots,x_k)$ for $\sigma_{\bp{B_1,\dots,B_k},\NS_{\omega_1}}$
a \emph{universal} formula $\theta_\phi(x_1,\dots,x_k)$ for 
$\sigma_{\bp{B_1,\dots,B_k, S_{\phi_{\UB,\NS_{\omega_1}}}},\NS_{\omega_1}}$
such that $T_{B_1,\dots,B_k}$ proves that
\[
\Pmax\Vdash 
[(H_{\omega_2}^{L(\UB)[\dot{G}]},\sigma_{\UB,\NS_{\omega_1}}^{L(\UB)[\dot{G}]})
\models\forall\vec{x}\;(\phi(x_1,\dots,x_k)\leftrightarrow\theta_\phi(x_1,\dots,x_k))]
\]
where $\dot{G}\in L(\UB)$ is the canonical $\Pmax$-name for the generic filter.
\end{enumerate}

\end{theorem}

\subsection{Proofs of Thm.~\ref{Thm:mainthm-1}, and of 
(\ref{thm:char(*)-modcomp-1})$\to$(\ref{thm:char(*)-modcomp-2})
of Thm.~\ref{Thm:mainthm-1bis}}

Theorem~\ref{Thm:mainthm-1}, (\ref{thm:char(*)-modcomp-1})$\to$(\ref{thm:char(*)-modcomp-2})
of Theorem~\ref{Thm:mainthm-1bis} are immediate corollaries
of the above theorem combined with Asper\`o and Schindler's proof that 
$\MM^{++}$ implies $\stUB$, and with Theorem \ref{thm:PI1invomega2}.


We start with the proof of
(\ref{thm:char(*)-modcomp-1})$\to$(\ref{thm:char(*)-modcomp-2}) of Thm.~\ref{Thm:mainthm-1bis}
assuming Thm. \ref{thm:keythmmodcompanHomega2} and Thm. \ref{thm:PI1invomega2}:

\begin{proof}
%
Assume $(V,\in)$ models $\stUB$. Then there is a $\Pmax$-filter $G\in V$ such that
$H_{\omega_2}^{L(\UB)[G]}=H_{\omega_2}^V$.
By Thm. \ref{thm:keythmmodcompanHomega2} and Robinson's test, 
we get that the first order $\sigma_{\UB,\NS_{\omega_1}}$-theory of 
$H_{\omega_2}^{L(\UB)[G]}$ is model complete.
By Levy's absoluteness \cite[Lemma 4.1]{VIATAMSTI}, $H_{\omega_2}^{L(\UB)[G]}$ is a $\Sigma_1$-elementary substructure of $V$ also according to the signature $\sigma_{\UB,\NS_{\omega_1}}$.
We conclude (by \cite[Thm. 3.18]{VIATAMSTI}), since the two theories share the same $\Pi_1$-fragment.
\end{proof}

The proof of the converse implication requires more information on 
$\bar{D}_{\UB,\NS_{\omega_1}}$ then what is conveyed in Thm. \ref{thm:keythmmodcompanHomega2}.
We defer it to a later stage.
%
%
%
%

\smallskip

We now prove Thm.~\ref{Thm:mainthm-1}:
\begin{proof}
By Thm. \ref{thm:keythmmodcompanHomega2} and Robinson's test, the
$\Pi_2$-sentences for $\sigma_{\lUB,\NS_{\omega_1}}$ which are $\ZFC^*_{\lUB}+\maxUB$-provably 
forced to 
hold in the $H_{\omega_2}$ of the 
generic extension of $L(\UB)$ by $\Pmax$ form a model complete theory.

Let us call $T^*_{\lUB,\NS_{\omega_1}}$ this model complete theory.

We now show that any model of $\ZFC^*_{\lUB,\NS_{\omega_1}}+\maxUB+$\emph{there is a supercompact cardinal} 
has the same 
$\Pi_1$-theory of some model of $T^*_{\lUB,\NS_{\omega_1}}$.
This suffices by \cite[Lemma 3.19]{VIATAMSTI}.

$\stUB$ holds in any model of $\MM^{++}$ by Schindler and Asper\`o's breakthrough \cite{ASPSCH(*)}.
It is a standard result that one can force $\MM^{++}$ over any model of 
$\ZFC+$\emph{there is a supercompact cardinal} \cite{FORMAGSHE}.

Let $\mathcal{M}$ be any model of 
$\ZFC+\maxUB+$\emph{there is a supercompact cardinal} 
and $\mathcal{N}$ be a model of $\MM^{++}$ obtained as a forcing extension
of $\mathcal{M}$ by the methods of \cite{FORMAGSHE}.

By  Thm. \ref{thm:PI1invomega2}, $\mathcal{N}$ has the same $\Pi_1$-theory of 
$\mathcal{M}$ according to the signature $\sigma_{\lUB,\NS_{\omega_1}}$.
Now $\mathcal{N}$
is a model of $\MM^{++}$ and therefore of $\stUB$, by \cite{ASPSCH(*)}.

Hence $H_{\omega_2}^{\mathcal{N}}$ is also (according to $\mathcal{N}$)
the $H_{\omega_2}$ of the generic extension
of $L(\UB)^{\mathcal{N}}$ by $\Pmax$.
Since $H_{\omega_2}^{\mathcal{N}}\prec_1\mathcal{N}$ also according to the signature
$\sigma_{\lUB,\NS_{\omega_1}}$, we conclude that $H_{\omega_2}^{\mathcal{N}}$
and $\mathcal{M}$ share the same $\Pi_1$-theory.
But $H_{\omega_2}^{\mathcal{N}}$ is a model of $T^*_{\lUB,\NS_{\omega_1}}$.

We are left with the proof of the equivalence between \ref{Thm:mainthm-1A},
\ref{Thm:mainthm-1B}, \ref{Thm:mainthm-1C}, \ref{Thm:mainthm-1D}, \ref{Thm:mainthm-1E}.

\begin{description}
\item[\ref{Thm:mainthm-1A}$\Longleftrightarrow$\ref{Thm:mainthm-1B}] By \cite[Lemma 3.19]{VIATAMSTI}
$T$ and $T^*$ have this property.
\item[\ref{Thm:mainthm-1A}$\Longrightarrow$\ref{Thm:mainthm-1E}]
By Levy's absoluteness  if $\mathcal{M}$ models 
\[
T_\forall+\ZFC^*_{\lUB,\NS_{\omega_1}}+\maxUB+\stUB
\]
$H_{\omega_2}\models T^*$. Therefore if $T^*\vdash \psi$,
$\mathcal{M}\models\psi^{H_{\omega_2}}$.
\item[\ref{Thm:mainthm-1E}$\Longrightarrow$\ref{Thm:mainthm-1D}]
By definition of $\stUB$.
\item[\ref{Thm:mainthm-1D}$\Longrightarrow$\ref{Thm:mainthm-1C}]
If $P$ forces $\MM^{++}$, by Asper\`o and Schindler result $P\Vdash\stUB$, hence 
$P\Vdash\psi^{H_{\omega_2}}$ by \ref{Thm:mainthm-1D}.
\item[\ref{Thm:mainthm-1C}$\Longrightarrow$\ref{Thm:mainthm-1B}]
Given some complete $S\supseteq T$, and a model $\mathcal{M}$ of $S$, 
find $\mathcal{N}$ forcing extension of $\mathcal{M}$ which models $\psi$.
By Thm. \ref{thm:PI1invomega2} and Levy's absoluteness \cite[Lemma 4.1]{VIATAMSTI}, 
$H_{\omega_2}^{\mathcal{N}}\models S_\forall$ and we are done.
\end{description}
\end{proof}

\subsection{Proofs of Thm.~\ref{thm:keythmmodcompanHomega2}
and Thm. \ref{thm:PI1invomega2}}
The rest of this section is devoted to the proof of Thm.~\ref{thm:keythmmodcompanHomega2}
and Thm. \ref{thm:PI1invomega2}.

Let us first set up the proper language and terminology in order to deal with the $\Pmax$-technology.

\subsubsection{Generic iterations of countable structures}
\begin{definition}\cite[Def. 1.2]{HSTLARSON}
Let $M$ be a transitive countable 
model of $ZFC$. 
Let $\gamma$ be an ordinal less than or equal to $\omega_1$. 
An iteration $\mathcal{J}$ of $M$ of length $\gamma$ 
consists of models $\ap{M_\alpha:\,\alpha \leq\gamma}$, sets $\ap{G_\alpha:\,\alpha< \gamma}$ 
and a commuting family of elementary embeddings 
\[
\ap{j_{\alpha\beta}: M_\alpha\to M_\beta:\, \alpha\leq\beta\leq\gamma}
\]
such that:
\begin{itemize}
\item
$M_0 = M$,
\item
each $G_\alpha$ is an $M_\alpha$-generic filter for 
$(\pow{\omega_1}/\NS_{\omega_1})^{M_\alpha}$,
\item
each $j_{\alpha\alpha}$ is the identity mapping,
\item
each $j_{\alpha\alpha+1}$ is the ultrapower embedding induced by $G_\alpha$,
\item
for each limit ordinal $\beta\leq\gamma$,
$M_\beta$ is the direct limit of the system
$\bp{M_\alpha, j_{\alpha\delta} :\, \alpha\leq\delta<\beta}$, and for each $\alpha<\beta$, $j_{\alpha\beta}$ is the induced embedding.
\end{itemize}
\end{definition}

We adopt the convention to denote an iteration $\mathcal{J}$ just
by $\ap{j_{\alpha\beta}:\, \alpha\leq\beta\leq\gamma}$, we also stipulate that
if $X$ denotes the domain of $j_{0\alpha}$, $X_\alpha$ or $j_{0\alpha}(X)$ will denote
the domain of $j_{\alpha\beta}$ for any $\alpha\leq\beta\leq\gamma$.

\begin{definition}
Let $A$ be  a universally Baire sets of reals.
$M$ is $A$-iterable if:
\begin{enumerate}
\item $M$ is transitive and such that $H_{\omega_1}^M$ is countable.
\item 
$M\models\ZFC+\NS_{\omega_1}$\emph{ is precipitous}.
\item
Any iteration 
\[
\bp{j_{\alpha\beta}:\alpha\leq\beta\leq\gamma}
\] 
of $M$ is well founded and such that 
$A\cap M_\beta=j_{\alpha\beta}(A\cap M_0)$ for all $\beta\leq\gamma$.
\end{enumerate}
\end{definition}

\subsubsection{Generic invariance of the universal fragment of the $\sigma_{\UB,\NS_{\omega_1}}$-theory of $V$}



We now prove Theorem \ref{thm:PI1invomega2} .
\begin{proof}
Let $\phi$ be a $\Pi_1$-sentence for $\sigma_{\mathcal{A},\NS_{\omega_1}}$ which holds in $V$.
Assume that for some forcing notion $P$, $\phi$ fails in $V[h]$ with $h$ $V$-generic for $P$.
By forcing over $V[h]$ with the appropriate stationary set preserving (in $V[h]$) 
forcing notion (using a Woodin cardinal $\gamma$ of $V[h]$), we may assume that $V[h]$ is extended to a
generic extension $V[g]$ such that $V[g]$ models $\NS_{\omega_1}$ is 
saturated\footnote{A result of Shelah whose outline can be found in \cite[Chapter XVI]{SHEPRO}, or \cite{woodinBOOK}, or in an \href{https://ivv5hpp.uni-muenster.de/u/rds/sat_ideal_better_version.pdf}{handout} of Schindler available on his webpage.}.
Since $V[g]$ is an extension of $V[h]$ by a stationary set preserving forcing and there are in $V[h]$ class many Woodin cardinals, we get that
$V[h]\sqsubseteq V[g]$ with respect to $\sigma_{\UB,\NS_{\omega_1}}$.
Since $\Sigma_1$-properties are upward absolute and $\neg\phi$ holds in $V[h]$, 
$\phi$ fails in $V[g]$ as well.

Let $\delta$ be inaccessible in $V[g]$ and let $\gamma>\delta$ be a Woodin cardinal.

Let $G$ be $V$-generic for $\tow{T}^{\omega_1}_\gamma$ (the countable tower $\mathbb{Q}_{<\gamma}$ according to \cite[Section 2.7]{STATLARSON})
and such that  $g\in V[G]$.
Let $j_G:V\to\Ult(V,G)$ be the induced ultrapower embedding.

Now remark that $V_\delta[g]\in \Ult(V,G)$ is $B^{V[G]}$-iterable for all 
$B\in \mathsf{UB}^{V}$ (since 
$V_\eta[g]\in \Ult(V,G)$ for all $\eta<\gamma$, and this suffices to check that $V_\delta[g]$ 
is $B^{V[G]}$-iterable for all $B\in\mathsf{UB}^V$, see \cite[Thm. 4.10]{HSTLARSON}).

By \cite[Lemma 2.8]{HSTLARSON} applied in $\Ult(V,G)$, there exists in $\Ult(V,G)$ an iteration 
$\mathcal{J}=\bp{j_{\alpha\beta}:\alpha\leq\beta\leq\gamma=\omega_1^{\Ult(V,G)}}$ of 
$V_\delta[g]$ such that
$\NS_{\omega_1}^{X_{\gamma}}=\NS_{\omega_1}^{\Ult(V,G)}\cap X_{\gamma}$, where 
$X_\alpha=j_{0\alpha}(V_\delta[g])$ for all $\alpha\leq\gamma=\omega_1^{\Ult(V,G)}$.

This gives that $X_{\gamma}\sqsubseteq \Ult(V,G)$ for $\sigma_{\UB,\NS_{\omega_1}}$.
Since $V_\delta[g]\models\neg\phi$, so does $X_{\gamma}$, by elementarity.
But $\neg\phi$ is a $\Sigma_1$-sentence, hence it is upward absolute for superstructures, therefore
$\Ult(V,G)\models\neg\phi$. This is a contradiction, since $\Ult(V,G)$ is elementarily equivalent to $V$ for
$\sigma_{\UB,\NS_{\omega_1}}$, and $V\models\phi$.

\smallskip

A similar argument shows that if $V$ models a $\Sigma_1$-sentence $\phi$ for 
$\sigma_{\UB,\NS_{\omega_1}}$
this will remain true in all of its generic extensions:

Assume $V[h]\models\neg\phi$
for some $h$ $V$-generic for some forcing notion $P\in V$. 
Let $\gamma>|P|$ be a Woodin cardinal, and let $g$ be $V$-generic for\footnote{$\tow{T}_\gamma$ is the full stationary tower of height $\gamma$ whose conditions are 
stationary sets in $V_\gamma$, denoted as $\mathbb{P}_{<\gamma}$ in \cite{STATLARSON},
see in particular \cite[Section 2.5]{STATLARSON}.}
 $\tow{T}_\gamma$ with $h\in V[g]$ and $\crit(j_g)=\omega_1^V$ (hence there is in $g$ some stationary set of $V_\gamma$ concentrating on countable sets). 
Then $V[g]\models\phi$ since:
\begin{itemize} 
\item
$V_\gamma\models\phi$, since $V_\gamma\prec_{1} V$ for $\sigma_{\UB,\NS_{\omega_1}}$
by \cite[Lemma 4.1]{VIATAMSTI};
\item
$V_\gamma^{\Ult(V,g)}=V_\gamma^{V[g]}$, since $V[g]$ models that 
$\Ult(V,g)^{<\gamma}\subseteq \Ult(V,g)$;
\item
$V_\gamma^{\Ult(V,g)}\models\phi$, by elementarity of $j_g$, since 
$j_g(V_\gamma)=V_\gamma^{\Ult(V,g)}$;
\item
$V_\gamma^{V[g]}\prec_{\Sigma_1}V[g]$ with respect to $\sigma_{\mathcal{A},\NS_{\omega_1}}$,
again by \cite[Lemma 4.1]{VIATAMSTI} applied in $V[g]$.
\end{itemize}

Now repeat the same argument as before to the $\Pi_1$-property $\neg\phi$,
with $V[h]$ in the place of $V$ and $V[g]$ in the place of $V[h]$. 
\end{proof}

Asper\'o and Veli\v{c}kovi\`c provided the following basic counterexample to the conclusion of the theorem
if large cardinal assumptions are dropped.
\begin{remark}
Let $\phi(y)$ be the $\Delta_1$-property in $\tau_{\NS_{\omega_1}}$
\[
\exists y (y=\omega_1 \wedge L_{y+1}\models y=\omega_1).
\]
Then $L$ models this property, while the property fails in any forcing extension of $L$ which collapses 
$\omega_1^L$ 
to become countable.
\end{remark}

\subsubsection{Proof of Thm. \ref{thm:keythmmodcompanHomega2}}
We now turn to the proof of Thm. \ref{thm:keythmmodcompanHomega2}.

What we will do first is to sketch a different proof of Thm.~\ref{thm:modcompanHomega1}.
This will give us the key intuition on how to define $\bar{D}_{\UB,\NS_{\omega_1}}$.

\subsubsection{A different proof of Thm.~\ref{thm:modcompanHomega1}.}

Let $M$ be a countable transitive model of 
$\ZFC+$\emph{there are class many Woodin cardinals}.
Then it will have its own version of Thm.~\ref{thm:modcompanHomega1}.
In particular it will model that the theory of $(H_{\omega_1}^M,\sigma_{\UB^M}^M)$ 
is model complete,
and also that $\UB^M$ is an $H_{\omega_1}$-closed family of universally Baire sets in $M$.

Now assume that there is a countable family $\UB_M$ of universally Baire sets in $V$
which is $H_{\omega_1}$-closed in $V$ and is such that
$\UB^M=\bp{B\cap M: B\in\UB_M}$.
Then 
\[
(H_{\omega_1}^M,\sigma_{\UB^M}^M)=(H_{\omega_1}^M,\bp{B\cap M: B\in\UB_M})\sqsubseteq
(H_{\omega_1}^V,\sigma_{\UB_M}^V)
\]
But $\UB_M$ being $H_{\omega_1}$-closed in $V$ entails that  the first order theory of
$(H_{\omega_1}^V,\sigma_{\UB_M}^V)$ is model complete.
In particular if $(H_{\omega_1}^M,\sigma_{\UB^M}^M)$ and $(H_{\omega_1}^V,\sigma_{\UB_M}^V)$
are elementarily equivalent, then 
\[
(H_{\omega_1}^M,\bp{B\cap M: B\in\UB_M})\prec
(H_{\omega_1}^V,\sigma_{\UB_M}^V).
\]
The setup described above is quite easy to realize (for example $M$ could the transitive collapse of
some countable $X\prec V_\theta$ for some large enough $\theta$); in particular for any
$a\in H_{\omega_1}$ and $B_1,\dots,B_k\in \UB$, we can find $M$ countable transitive model of a suitable fragment of $\ZFC$ with $a\in H_{\omega_1}^M$
and $\UB_M\supseteq \bp{B_1,\dots,B_k}$ countable and $H_{\omega_1}$-closed family of $\UB$-sets
in $V$, such that:
\begin{itemize}
\item $\UB^M=\bp{B\cap M: B\in\UB_M}$;
\item the first order theory $T_{\UB_M}$ of $(H_{\omega_1}^V,\sigma_{\UB_M}^V)$ is model complete;
\item $(H_{\omega_1}^M,\bp{B\cap M: B\in\UB_M})$ models $T_{\UB_M}$.
\end{itemize}
Letting $B_M=\prod\UB_M$, $(H_{\omega_1}\cup\UB,\in)$ is able to compute correctly whether 
$B_M$ encodes a set $\UB_M$ such that the pair $(\UB_M,M)$ satisfies the above list of requirements;
here we use crucially the fact that being a model complete theory is a $\Delta_0$-property, and also
that it is possible to encode
the structure
$(H_{\omega_1}^V,\sigma_{\UB^M}^V)$ in a single universally Baire set\footnote{See \cite[Def. 2.2]{VIATAMSTI} for the definition of $\WFE_\omega$ and $\Cod_\omega$.} 
(for example $\WFE_\omega\times B_M$).

In particular $(H_{\omega_1}\cup\UB,\in)$ correctly computes 
the set $D_{\UB}$ of $M\in H_{\omega_1}$ such that there exists a universally Baire set
$B_M=\prod\UB_M$ with the property that the pair $(M,\UB_M)$ realizes the above set of requirements.
By $\maxUB$, $\bar{D}_{\UB}=\Cod_\omega^{-1}[D_{\UB}]$ is a universally Baire set $\bar{D}_{\UB}$.

Note moreover that $\bar{D}_{\UB}$ is defined by a $\in$-formula $\phi_\UB(x)$ 
in no extra parameters;
in particular for any model $\mathcal{W}=(W,E)$ of $\ZFC+\maxUB$, we can define 
$\bar{D}_{\UB}$ in
$\mathcal{W}$ and all its properties outlined above will hold relativized to $\mathcal{W}$.

For fixed universally Baire sets $B_1,\dots,B_k$ the set $D_{\UB,B_1,\dots,B_k}$ of 
$M\in D_\UB$ such that there is a witness $\UB_M$ of $M\in D_{\UB}$ with $B_1,\dots,B_k\in \UB_M$ is 
also definable in 
\[
(H_{\omega_1}\cup\UB,\in)
\] 
in parameters $B_1,\dots,B_k$.
Hence by $\maxUB$ $\Cod_\omega^{-1}[D_{\UB,B_1,\dots,B_k}]=\bar{D}_{\UB,B_1,\dots,B_k}$ is universally Baire (note as well that $\bar{D}_{\UB,B_1,\dots,B_k}$ belongs to any 
$L(\UB)$-closed family $\mathcal{A}$ containing $B_1,\dots,B_k$).

Now take any $\Sigma_1$-formula $\phi(\vec{x})$ for $\sigma_{\UB}$ mentioning just the universally Baire
predicates $B_1,\dots,B_k$.
It doesn't take long to realize that for all $\vec{a}$ in $H_{\omega_1}$
\[
(H_{\omega_1}^V,\sigma_{\UB}^V)\models\phi(\vec{a})
\]
if and only if
\[
(H_{\omega_1}^M,\sigma_{\UB_M}^M)\models\phi(\vec{a})
\text{\emph{ for all $M\in D_{\UB,B_1,\dots,B_k}$ with 
$\vec{a}\in H_{\omega_1}^M$}. }
\]


But $\bar{D}_{\UB,B_1,\dots,B_k}$ is universally Baire, so the above can be formulated also as:
\[
\forall r\in\bar{D}_{\UB,B_1,\dots,B_k}[\vec{a}\in H_{\omega_1}^{\Cod(r)}\rightarrow 
(H_{\omega_1}^{\Cod(r)},\sigma_{\UB_{\Cod(r)}}^{\Cod(r)})\models\phi(\vec{a})].
\]
The latter is a $\Pi_1$-sentence in the universally Baire parameter $\bar{D}_{\UB,B_1,\dots,B_k}$.

This is exactly a proof that Robinson's test applies to the $\sigma_{\UB}$-first order theory of $H_{\omega_1}$ assuming
$\maxUB$; i.e. we have briefly sketched a different (and much more convoluted) 
proof of the conclusion of 
Thm.~\ref{thm:modcompanHomega1} (using as hypothesis Thm.~\ref{thm:modcompanHomega1} itself).
What we gained however is an insight on how to prove Theorem~\ref{thm:keythmmodcompanHomega2}.

We will consider the set $D_{\NS_{\omega_1},\UB}$ 
of $M\in D_{\UB}$ such that:
\begin{itemize} 
\item
$(M,\NS_{\omega_1}^M)$ is a $\Pmax$-precondition 
which is $B$-iterable for all $B\in \UB_M$ (according to \cite[Def. 4.1]{HSTLARSON});
\item
$j_{0\omega_1}$ is a $\Sigma_1$-elementary embedding of $H_{\omega_2}^M$ into $H_{\omega_2}^V$
for $\sigma_{\UB_M,\NS_{\omega_1}}$ whenever
$\mathcal{J}=\bp{j_{\alpha\beta}:\alpha\leq\beta\leq\omega_1}$ is an iteration of $M$ with 
$j_{0\omega_1}(\NS_{\omega_1}^M)=\NS_{\omega_1}^V\cap j_{0\omega_1}(H_{\omega_2}^M)$.
\end{itemize}

It will take a certain effort to prove that  assuming $(*)$-$\UB$:
\begin{itemize}
\item for any $A\in H_{\omega_2}$ and
$B\in\UB$, we can find $M\in D_{\NS_{\omega_1},\UB}$ with $B\in \UB_M$, $a\in H_{\omega_2}^M$,
and an iteration $\mathcal{J}=\bp{j_{\alpha\beta}:\alpha\leq\beta\leq\omega_1}$ of $M$ with 
$j_{0\omega_1}(\NS_{\omega_1})=\NS_{\omega_1}^V\cap j_{0\omega_1}(H_{\omega_2}^M)$  such that
$j_{0\omega_1}(a)=A$.
\item
$D_{\NS_{\omega_1},\UB}$ is correctly computable in $(H_{\omega_1}\cup\UB,\in)$.
\end{itemize}
But this effort will pay off since we will then be able to prove the model completeness of the theory
\[
(H_{\omega_2},\sigma_{\UB,\NS_{\omega_1}}^V)
\]
using Robinson's test with $\Cod_\omega^{-1}[D_{\NS_{\omega_1},\UB}]$ in the place of $\bar{D}_{\UB}$ 
and replicating  in the new setting
what was sketched before for $(H_{\omega_1},\sigma_{\UB,\NS_{\omega_1}}^V)$.

We now get into the details.

\subsubsection{$\UB$-correct models} 
 
\begin{notation}
Given a countable family 
$\mathcal{A}=\bp{B_n:n\in\omega}$ of universally Baire sets with each $B_n\subseteq (2^{\omega})^{k_n}$,
we say that $B_\mathcal{A}=\prod_{n\in\omega}B_n\subseteq \prod_n(2^{\omega})^{k_n}$ is a code for 
$\bp{B_n:n\in\omega}$.

Clearly $B_\mathcal{A}$ is a universally Baire subset of the Polish space $\prod_n(2^{\omega})^{k_n}$.
\end{notation}

\begin{definition}
$T_\mathsf{UB}$ is the $\in$-theory of 
\[
(H_{\omega_1},\sigma_{\UB}).
\]


A transitive model of $\ZFC$ $(M,\in)$ is $\mathsf{UB}$-correct if 
there is a  $H_{\omega_1}$-closed (in $V$) family
$\mathsf{UB}_M$ of universally Baire sets in $V$
such that:
\begin{itemize}
\item The map 
\begin{align*}
\Theta_M:&\UB_M\to M\\
&A\mapsto A\cap M
\end{align*}
is injective.
\item
$(M,\in)$ models that $\bp{A\cap M: A\in \mathsf{UB}_M}$ is the family of universally Baire subsets of $M$.

\item Letting $T_{\UB_M}$ be the theory of $(H_{\omega_1}, \tau_{\mathsf{ST}}^V,\mathsf{UB}_M)$
\[
(H_{\omega_1}^M, \tau_{\mathsf{ST}}^M,A\cap M: A\in \mathsf{UB}_M)\models T_{\UB_M}.
\]
\item If $M$ is countable,
$M$ is $A$-iterable for all $A\in \mathsf{UB}_M$.
\end{itemize}
\end{definition}

Remark (by Thm.~\ref{thm:modcompanHomega1}) that if $M$ is $\UB$-correct, $T_{\UB_M}$ is model complete, 
since $\UB_M$ is (in $V$) a $H_{\omega_1}$-closed family of 
universally Baire sets.

\begin{notation}
$D_{\UB}$ denotes the set of countable $\mathsf{UB}$-correct $M$; $\bar{D}_\UB=\Cod_\omega^{-1}[D_\UB]$.

For each $M$ $\UB_M$ is a witness that $M\in D_\UB$ and $B_{\UB_M}=\prod\UB_M$ is a universally Baire coding this 
witness\footnote{The Fact below shows that the map $M\mapsto (\UB_M,B_{\UB_M})$, can be chosen in $L(\UB)$.}.

For universally Baire sets $B_1,\dots,B_k$, $E_{\UB,B_1,\dots,B_k}$
denotes the set of $M\in D_\UB$ with $B_1,\dots,B_k\in \UB_M$ for some witness $\UB_M$ that $M\in D_\UB$;
$\bar{E}_{\UB,B_1,\dots,B_k}=\Cod_\omega^{-1}[E_{\UB,B_1,\dots,B_k}]$.
\end{notation}

\begin{fact}  
$(V,\in)$ models $M$\emph{ is countable and $\UB$-correct as witnessed by $\UB_M$}
if and only if so does 
$(H_{\omega_1}\cup{\UB},\in)$. 

Consequently 
the set $D_{\mathsf{UB}}$ of countable $\mathsf{UB}$-correct $M$ is properly computed in
$(H_{\omega_1}\cup\UB,\in)$.

Therefore assuming $\maxUB$ 
\[
\bar{D}_\UB=\mathrm{Cod}^{-1}[D_{\mathsf{UB}}]
\] 
is universally Baire. 

Moreover there is in $L(\UB)$ a definable map $M\mapsto \UB_M$ assigning to each $M\in D_\UB$ 
a countable family $\UB_M$ witnessing it.

The same holds for $\bar{E}_{\UB,B_1,\dots,B_k}$ for given universally Baire sets $B_1,\dots,B_k$.
\end{fact}
\begin{proof}
The first part follows almost immediately by the definitions, since the assertion in parameters $B,M$:
\begin{quote}
\emph{
$B=\prod_{n\in\omega}B_n$ codes a $H_{\omega_1}$-closed family 
$\UB_M=\bp{B_n:n\in\omega}$ of sets such that 
\begin{itemize}
\item $M$ is $A$-iterable for all $A\in \UB_M$,
\item $M$ models that $\bp{A\cap M: A\in \UB_M}$ is its family of universally 
Baire sets and is $H_{\omega_1}$-closed,
\item
$(H_{\omega_1}^M,\tau_{\ST}^M,\bp{A\cap M: A\in \mathsf{UB}_M})$ models
$T_{\UB_M}$.
\end{itemize}
}
\end{quote}
gets the same truth value in $(V,\in)$ and in $(H_{\omega_1}\cup \UB,\in)$.
%

We conclude that $D_{\mathsf{UB}}$ has the same extension in 
$(V,\in)$ and in $(H_{\omega_1}\cup\UB,\in)$.
By $\maxUB$ $\bar{D}_\UB$ is universally Baire.

The existence of class many Woodin cardinals grants
that we can always find\footnote{For example by \cite[Thm. 36.9]{kechris:descriptive} and \cite[Thm. 3.3.14, Thm. 3.3.19]{STATLARSON}.} 
a universally Baire uniformization of the 
universally Baire relation 
on $\bar{D}_\UB\times 2^\omega$ given by the pairs $\ap{r,B}$ such that $B=\prod\bp{B_n:n\in\omega}$ witnesses
$\Cod_\omega(r)\in D_\UB$ .

The same argument can be replicated for  $\bar{E}_{\UB,B_1,\dots,B_k}$.
\end{proof}


\begin{lemma}\label{lem:UBcorr}
Assume $\NS_{\omega_1}$ is precipitous and there are class many 
Woodin cardinals in $V$.
Let $\delta$ be an inaccessible cardinal in $V$ and $G$ be 
$V$-generic for $\Coll(\omega,\delta)$.
Then $V_\delta$ is $\UB^{V[G]}$-correct in $V[G]$ as witnessed by $\bp{B^{V[G]}:B\in \UB^V}$.
\end{lemma}
\begin{proof}
Let in $V$ 
$\bp{(T_A,S_A): A\in\UB^V}$ be an enumeration of pairs of trees $S_A,U_A$ on $\omega\times\gamma$ for 
a large enough inaccessible $\gamma>\delta$
such that $T_A,S_A$ projects to complements in $V[G]$ and $A$ is the projection of $T$.
Then $A^{V[G]}$ is correctly computed as the projection of $T_A$ in $V[G]$ for any $A\in\UB^V$.
By Fact~\ref{fac:keyfacHomega1clos} and \cite[Thm. 4.7]{VIATAMSTI} 
\[
(H_{\omega_1}^V, \tau_{\mathsf{ST}}^V,\mathsf{UB}^V)\prec (H_{\omega_1}^{V[G]}, \tau_{\mathsf{ST}}^{V[G]},A^{V[G]}: 
A\in\mathsf{UB}^{V}),
\]
$\bp{A^{V[G]}: A\in\mathsf{UB}^{V}}$ is a $H_{\omega_1}$-closed family of 
universally Baire sets in $V[G]$, and 
$T_{\UB^V}$ is also the theory of $(H_{\omega_1}^{V[G]}, \tau_{\mathsf{ST}}^{V[G]},A^{V[G]}: 
A\in\mathsf{UB}^{V})$.

To conclude that $\bp{A^{V[G]}: A\in\mathsf{UB}^{V}}$ witnesses in $V[G]$ that $V_\delta$ is $\UB^{V[G]}$-correct in 
$V[G]$ it remains to argue that 
$V_\delta$ is $B^{V[G]}$-iterable for any $B\in\UB^V$.

Let $\mathcal{J}$ be any iteration of $V_\delta$ in $V[G]$.
%
%
%
%
%
Then  by standard results on iterations
(see \cite[Lemma 1.5, Lemma 1.6]{HSTLARSON})
$\mathcal{J}$ extends uniquely to an iteration $\bar{\mathcal{J}}$ of $V$ in $V[G]$ such that
\begin{itemize}
\item
$\bar{j}_{\alpha\beta}$ is a proper extension of $j_{\alpha\beta}$ for all $\alpha\leq\beta\leq\gamma$
(i.e. letting $\bar{V}_\alpha=\bar{j}_{0\alpha}(V)$, we have that
$j_{0\alpha}(V_\delta)$ is the rank initial segments of elements of $\bar{V}_\alpha$ of rank less than
$\bar{j}_{0\alpha}(\delta)$).
\item
$\bar{\mathcal{J}}$ is a well defined iteration of transitive structures. 
\end{itemize}
In particular this shows that $V_\delta$ is iterable in $V[G]$.

Now fix $B\in \UB^V$. We must argue that $j_{0\alpha}(B)=B^{V[G]}\cap \bar{j}_{0\alpha}(V)$.
To simplfy notation we assume $B\subseteq 2^\omega$.
Let $(T_B,S_B)$ be the pair of trees selected in $V$ to define $B^{V[G]}$.

Then  
\[
\bar{j}_{0\alpha}(V)\models (\bar{j}_{0\alpha}(T_B), \bar{j}_{0\alpha}(S_B))
\]
projects to complements; clearly $\bar{j}_{0\alpha}[T_B]\subseteq \bar{j}_{0\alpha}(T_B)$,
$\bar{j}_{0\alpha}[S_B]\subseteq \bar{j}_{0\alpha}(S_B)$.
Let $p:(\gamma\times 2)^{\omega}\to 2^\omega$ be the projection map. 

This gives that
\[
B^{V[G]}\cap \bar{j}_{0\alpha}(V)=p[[T_B]]\cap \bar{j}_{0\alpha}(V)=p[[\bar{j}_{0\alpha}[T_B]]]\cap \bar{j}_{0\alpha}(V)\subseteq
p[[\bar{j}_{0\alpha}(T_B)]]\cap \bar{j}_{0\alpha}(V)=\bar{j}_{0\alpha}(B).
\]
Similarly
\[
((2^{\omega})^{V[G]}\setminus B^{V[G]})\cap \bar{j}_{0\alpha}(V)=p[[S_B]]\cap \bar{j}_{0\alpha}(V)\subseteq
p[[\bar{j}_{0\alpha}(S_B)]]\cap \bar{j}_{0\alpha}(V)=\bar{j}_{0\alpha}((2^\omega)^V\setminus B).
\]
By elementarity
\[
\bar{j}_{0\alpha}((2^\omega)^V\setminus B)\cup \bar{j}_{0\alpha}(B)=(2^\omega)\cap \bar{j}_{0\alpha}(V).
\]
These three conditions can be met only if
\[
B^{V[G]}\cap \bar{j}_{0\alpha}(V)=\bar{j}_{0\alpha}(B).
\]

Since $\mathcal{J}$ and $B$ were chosen arbitrarily,
we conclude that $V_\delta$ is $B^{V[G]}$-iterable in $V[G]$ for all $B\in\UB^V$.

Hence $V_\delta$ is $\UB^{V[G]}$-correct in $V[G]$ as witnessed by $\bp{A^{V[G]}: A\in \UB^V}$.
\end{proof}

\begin{definition}
Given $M,N$ iterable structures, 
$M\geq N$ if $M\in (H_{\omega_1})^N$ and there is  an iteration
\[
\mathcal{J}=\bp{j_{\alpha\beta}:\,\alpha\leq\beta\leq\gamma=(\omega_1)^N}
\]
of $M$ with $\mathcal{J}\in N$
such that
\[
\NS_{\gamma}^{M_\gamma}=
\NS_{\gamma}^N\cap M_{\gamma}.
\]
\end{definition}

\begin{fact}\label{fac:densityUBcorrect}
$(\maxUB)$
Assume $\NS_{\omega_1}$ is precipitous and $\maxUB$ holds.
Then for any iterable $M$ and $B_1,\dots,B_k\in \UB$, there is an $\UB$-correct $N\geq M$ with $B_1,\dots,B_k\in \UB_N$.
\end{fact}
\begin{proof}
The assumptions grant that whenever $G$ is $\Coll(\omega,\delta)$-generic for $V$,
in $V[G]$ $V_\delta$ is $\UB^{V[G]}$-correct in $V[G]$ (i.e. Lemma \ref{lem:UBcorr}).

By \cite[Lemma 2.8]{HSTLARSON}, for any iterable $M\in H_{\omega_1}^V$ there is in $V$ an iteration
 $\mathcal{J}=\bp{j_{\alpha\beta}:\alpha\leq\beta\leq \omega_1^V}$ of $M$ such that
 $\NS_{\omega_1}^V\cap M_{\omega_1}=\NS_{\omega_1}^{M_{\omega_1}}$.
 
By $\maxUB$
\[
(H_{\omega_1}^{V}\cup\UB^{V},\in) \prec 
(H_{\omega_1}^{V[G]}\cup\UB^{V[G]},\in).
\]
 Therefore  we have that in $V[G]$ $\bar{E}_{\UB,B_1,\dots,B_k}^{V[G]}$ is exactly 
 $\bar{E}_{\UB,B_1^{V[G]},\dots,B_k^{V[G]}}$.
 
Hence for each iterable $M\in H_{\omega_1}^V$ and $B\in \UB^V$
 \[
 (H_{\omega_1}^{V[G]},\sigma_{\UB^V}^{V[G]})\models\exists \,N\geq M\text{ $\UB^{V[G]}$-correct with $B^{V[G]}$ in $\UB_N$},
 \]
 as witnessed by $N=V_\delta$, i.e. 
 \[
 (H_{\omega_1}^{V[G]},\sigma_{\UB^V}^{V[G]})\models\exists \,N\geq M \,(\bar{E}_{\UB,B_1,\dots,B_k}^{V[G]}(N)).
 \]

Since
  \[
 (H_{\omega_1}^{V},\sigma_{\UB^V}^{V})\prec (H_{\omega_1}^{V[G]},\sigma_{\UB^V}^{V[G]}),
 \]
we get that for every iterable $M\in H_{\omega_1}$ and $B\in \UB^V$
 \[
 (H_{\omega_1}^{V},\sigma_{\UB^V}^{V})\models\exists \,N\geq M\,(\bar{E}_{\UB,B_1,\dots,B_k}(N)).
 \]
 The conclusion follows.
\end{proof}

\begin{lemma}
$(\maxUB)$

Let $M\geq N$ be both $\UB$-correct structures, with
$\UB_N$ a witness of $N$ being $\UB$-correct  such that $\bar{D}_{\mathsf{UB}}\in\UB_N$.
Then
\[
(H_{\omega_1}^{M},\tau_{\mathsf{ST}}^M,A\cap M: A\in\UB_M)\prec
(H_{\omega_1}^{N},\tau_{\mathsf{ST}}^N,A\cap N: A\in\UB_M).
\]
\end{lemma}
\begin{proof}
Since $N\leq M$, and $N$ is $\UB$-correct with $\bar{D}_{\mathsf{UB}}\in\UB_N$
we get that 
\[
(H_{\omega_1}^N,\sigma_{\UB_N}^N)\models M\in D_{\mathsf{UB}}\cap N=\Cod[\bar{D}_{\mathsf{UB}}\cap N],
\]
since
\[
(H_{\omega_1}^N,\sigma_{\UB_N}^N)\prec (H_{\omega_1}^V,\sigma_{\UB_N}^V)
\]
and
\[
(H_{\omega_1}^V,\sigma_{\UB_N}^V)\models M\in D_{\mathsf{UB}}=\Cod[\bar{D}_{\mathsf{UB}}].
\]

Therefore $N$ models that there is a countable set $\UB_M^N=\bp{B_n^N:n\in\omega}\in N$ 
coded by the universally Baire set in $N$ $B_{\UB_M}^N=\prod_{n\in\omega}B_n^N$ such that 
$\bp{A\cap M:A\in \UB_M^N}\in M$ defines the family of universally Baire sets according to $M$, and such that $N$ models that
$M$ is $B^N$ iterable for all $B^N\in\UB_M^N$.
Now $N$ models that 
\[
\prod_{n\in\omega}B_n^N
\]
is a universally Baire set on the appropriate product space. 
Therefore there is $B\in \UB_N$ such that
$B\cap N=\prod_{n\in\omega}B_n^N$.
Clearly $\UB_M^N$ is computable from $B\cap N$.
Since
\[
(H_{\omega_1}^N,\sigma_{\UB_N}^N)\prec (H_{\omega_1}^V,\sigma_{\UB_N}^V).
\]
we conclude that in $V$ $B=\prod_{n\in\omega}B_n$ codes a set $\UB_M=\bp{B_n:n\in\omega}$ 
witnessing that $M$ is $\UB$-correct.

This gives that $\UB_M\subseteq\UB_N$.

Therefore 
$(H_{\omega_1}^N,\sigma_{\UB_M}^N)$
is also a model of
$T_{\UB_M}$.
By model completeness of $T_{\UB_M}$ we conclude that
\[
(H_{\omega_1}^M,\sigma_{\UB_M}^M)\prec (H_{\omega_1}^N,\sigma_{\UB_M}^N),
\]
as was to be shown.
\end{proof}

\subsection{Three characterizations of $(*)$-$\UB$}


\begin{definition}
For a $\UB$-correct $M$ with witness $\UB_M$, 
$T_{\NS_{\omega_1},\UB_M}$ is the 
$\sigma_{\UB_M,\NS_{\omega_1}}$-theory of $H_{\omega_2}^M$.

\smallskip

A $\UB$-correct $M$ is \emph{$(\NS_{\omega_1},\UB)$-ec}
if $(M,\in)$ models that $\NS_{\omega_1}$ is precipitous and 
there is a witness $\UB_M$ that $M$ is $\UB$-correct with the following property:   

\begin{quote}
Assume an iterable $N\geq M$ is $\UB$-correct with witness $\UB_N$ such that 
$B_{\UB_M}\in \UB_N$ (so that $\UB_M\subseteq \UB_N$).

Then for all iterations
\[
\mathcal{J}=\bp{j_{\alpha\beta}:\alpha\leq\beta\leq\gamma=\omega_1^N}
\] 
in $N$ witnessing
$M\geq N$, we have that $j_{0\gamma}$ defines a $\Sigma_1$-elementary embedding of
\[
(H_{\omega_2}^{M},\tau_{\mathsf{ST}}^{M},B\cap M: B\in\UB_M,\NS_{\omega_1}^{M})
\]
into 
\[
(H_{\omega_2}^N,\tau_{\mathsf{ST}}^N,B\cap N: B\in\UB_M,\NS_{\omega_1}^N).
\]
\end{quote}
\end{definition}

\begin{remark}\label{rmk:maxUBec}

A crucial observation is that
\emph{``$x$ is $(\NS_{\omega_1},\UB)$-ec''}
is a property correctly definable in $(H_{\omega_1}\cup\UB,\in)$.
Therefore (assuming $\maxUB$)
\[
D_{\NS_{\omega_1},\UB}=\bp{M\in H_{\omega_1}: \, M\text{ is $(\NS_{\omega_1},\UB)$-ec}}
\]
is such that $\bar{D}_{\NS_{\omega_1},\UB}=\Cod_\omega^{-1}[D_{\NS_{\omega_1},\UB}]$ 
is a universally Baire set in $V$.
Moreover letting for $V[G]$ a generic extension of $V$
\[
D_{\NS_{\omega_1},\UB^{V[G]}}=\bp{M\in H_{\omega_1}^{V[G]}: \, M\text{ is $(\NS_{\omega_1},\UB^{V[G]})$-ec}},
\]
we have that 
\[
\bar{D}_{\NS_{\omega_1},\UB}^{V[G]}=\Cod_\omega^{-1}[D_{\NS_{\omega_1},\UB^{V[G]}}].
\]
\end{remark}

\begin{theorem}\label{thm:char(*)}
Assume $V$ models $\maxUB$.
The following are equivalent:
\begin{enumerate}
\item\label{thm:char(*)-1}
Woodin's axiom $(*)$-$\mathsf{UB}$ holds
(i.e. 
$\NS_{\omega_1}$ is saturated,
and 
there is an $L(\mathsf{UB})$-generic filter $G$ for $\mathbb{P}_{\mathrm{max}}$
such that $L(\mathsf{UB})[G]\supseteq\pow{\omega_1}^V$).
\item\label{thm:char(*)-2}
Let $\delta$ be inaccessible.
Whenever $G$ is $V$-generic for $\Coll(\omega,\delta)$,
$V_\delta$ is $(\NS_{\omega_1},\UB^{V[G]})$-ec in $V[G]$.

\item \label{thm:char(*)-3}
$\NS_{\omega_1}$ is precipitous and
for all $\vec{A}\in H_{\omega_2}$, $B\in\UB$, there is an $(\NS_{\omega_1},\UB)$-ec $M$
with witness $\UB_M$,
and an iteration $\mathcal{J}=\bp{j_{\alpha\beta}:\,\alpha\leq\beta\leq\omega_1}$ of $M$ 
such that:
\begin{itemize}
\item $A\in M_{\omega_1}$,
\item $B\in\UB_M$,
\item $\NS_{\omega_1}^{M_{\omega_1}}=\NS_{\omega_1}\cap M_{\omega_1}$.
\end{itemize}
\end{enumerate}
\end{theorem}

Theorem \ref{thm:char(*)} is the key to the proofs of Theorem~\ref{thm:keythmmodcompanHomega2}
and to the missing implication in the proof of Theorem~\ref{Thm:mainthm-1bis}.

\subsubsection{Proof of Theorem~\ref{thm:keythmmodcompanHomega2}}

The theorem is an immediate corollary of the following:

\begin{lemma}\label{lem:keylemmodcomp(*)}

Let $B_1,\dots,B_k$ be new predicate symbols and 
$T_{B_1,\dots,B_k,\NS_{\omega_1}}$ 
be the $\tau_{\NS_{\omega_1}}\cup\bp{B_1,\dots,B_k}$-theory 
$\ZFC^*_{\NS_{\omega_1}}+\maxUB$ enriched with
the sentences asserting that $B_1,\dots,B_k$ are universally Baire sets.

Let $E_{B_1,\dots,B_k}$ consists of the set of
$M\in D_{\NS_{\omega_1},\UB}$ such that:
\begin{itemize}
\item
$M$ is $B_j$-iterable for all $j=1,\dots,k$;
\item 
there is $\UB_M$ witnessing $M\in D_{\NS_{\omega_1},\UB}$ with 
$B_j\in\UB_M$ for all $j$.
\end{itemize}
Let also $\bar{E}_{B_1,\dots,B_k}=\Cod_\omega^{-1}[E_{B_1,\dots,B_k}]$. 

Then $T_{B_1,\dots,B_k,\NS_{\omega_1}}$ proves that
$\bar{E}_{B_1,\dots,B_k}$ is universally Baire. 

Moreover let $T_{B_1,\dots,B_k,\bar{E}_{B_1,\dots,B_k},\NS_{\omega_1}}$ be the  natural extension of  
$T_{B_1,\dots,B_k,\NS_{\omega_1}}$
adding a predicate symbol for $\bar{E}_{B_1,\dots,B_k}$ 
and the axiom forcing its intepretation to be its definition.

Then $T_{B_1,\dots,B_k,\bar{E}_{B_1,\dots,B_k},\NS_{\omega_1}}$
models that every $\Sigma_1$-formula $\phi(\vec{x})$ for the signature
$\tau_{\NS_{\omega_1}}\cup\bp{B_1,\dots,B_k}$ is equivalent
to a $\Pi_1$-formula $\psi(\vec{x})$ in the signature
$\tau_{\NS_{\omega_1}}\cup\bp{B_1,\dots,B_k, \bar{E}_{B_1,\dots,B_k}}$.
\end{lemma}

\begin{proof}
$\bar{E}_{B_1,\dots,B_k}$ is universally Baire by $\maxUB$,
since $E_{B_1,\dots,B_k}$
is definable in $(H_{\omega_1}\cup\mathsf{UB},\in)$ with parameters the universally Baire sets
$B_1,\dots,B_k,\bar{D}_{\NS_{\omega_1},\UB}$.

Given any $\Sigma_1$-formula $\phi(\vec{x})$ for $\tau_{\NS_{\omega_1}}\cup\bp{B_1,\dots,B_k}$
mentioning the universally Baire predicates $B_1,\dots,B_k$, we want
to find a universal formula $\psi(\vec{x})$ such that
\[
T_{\bp{B_1,\dots,B_k, \bar{E}_{B_1,\dots,B_k}},\NS_{\omega_1}}\models \forall\vec{x}(\phi(\vec{x})\leftrightarrow \psi(\vec{x})).
\]

Let $\psi(\vec{x})$ be the formula
asserting:
 
\begin{quote}
\emph{For all $M\in E_{B_1,\dots,B_k}$, for all iterations 
$\mathcal{J}=\bp{j_\alpha\beta:\alpha\leq\beta\leq\omega_1}$ of $M$ such that:}
\begin{itemize}
\item
$\vec{x}=j_{0\omega_1}(\vec{a})$\emph{ for some }$\vec{a}\in M$,
\item
$\NS_{\omega_1}^{j_{0\omega_1}(M)}=\NS_{\omega_1}\cap j_{0\omega_1}(M)$,
\end{itemize}
\[
(H_{\omega_2}^{M},\sigma_{\UB_{M},\NS_{\omega_1}}^{M})\models\phi(\vec{a}).
\]
\end{quote}

More formally:
\begin{align*}
\forall r\, \forall \mathcal{J}&\{&\\
&[&\\
&(r\in \bar{E}_{B_1,\dots,B_k})\wedge&\\
&\wedge\mathcal{J}=\bp{j_\alpha\beta:\alpha\leq\beta\leq\omega_1} \text{ is an iteration of }\Cod(r)\wedge&\\
&\wedge\NS_{\omega_1}^{j_{0\omega_1}(\Cod(r))}=\NS_{\omega_1}\cap j_{0\omega_1}(\Cod(r))\wedge&\\
&\wedge \exists\vec{a}\in\Cod(r)\,(\vec{x}=j_{0\omega_1}(\vec{a}))&\\
&]&\\
&\rightarrow&\\
&(H_{\omega_2}^{\Cod(r)},\sigma_{\UB_{\Cod(r)},\NS_{\omega_1}}^{\Cod(r)})\models\phi(\vec{a})&\\
&\}.&
\end{align*}
The above is a $\Pi_1$-formula  for 
$\tau_{\NS_{\omega_1}}\cup\bp{B_1,\dots,B_k, \bar{E}_{B_1,\dots,B_k}}$. 

 
(We leave to the reader to check that the property 
\begin{quote}
\emph{$\mathcal{J}=\bp{j_\alpha\beta:\alpha\leq\beta\leq\omega_1}$ is an iteration of $M$ such that 
$\NS_{\omega_1}^{j_{0\omega_1}(M)}=\NS_{\omega_1}\cap j_{0\omega_1}(M)$}
\end{quote}
is definable by a $\Delta_1$-property in parameters $M,\mathcal{J}$ in the signature
$\tau_{\NS_{\omega_1}}$).

Now it is not hard to check that:
\begin{claim}
For all $\vec{A}\in H_{\omega_2}$
\[
(H_{\omega_2}^V,\tau_{\NS_{\omega_1}}^V,B_1,\dots,B_k)\models\phi(\vec{A})
\]
if and only if 
\[
(H_{\omega_2},\tau_{\NS_{\omega_1}}^V,B_1,\dots,B_k, \bar{E}_{B_1,\dots,B_k})
\models\psi(\vec{A}).
\]
\end{claim}
\begin{proof}
\emph{}

\begin{description}
\item[$\psi(\vec{A})\rightarrow \phi(\vec{A})$]
Take any
$M$ and $\mathcal{J}$ satisfying the premises of the implication in $\psi(\vec{A})$, 
Then $(H_{\omega_2}^M,\tau_{\NS_{\omega_1},\UB^M}^M)\models\phi(\vec{a})$
for some $\vec{a}$ such that $j_{0,\omega_1}(\vec{a})=\vec{A}$ and
$B_j\cap M_{\omega_1}=j_{0\omega_1}(B_j\cap M)$ for all $j=1,\dots,k$.

Since $\Sigma_1$-properties are upward absolute and
$(M_{\omega_1},\tau_{\NS_{\omega_1}}^{M_{\omega_1}},B_j\cap M_{\omega_1}:j=1,\dots,k)$
is a $\tau_{\NS_{\omega_1}}\cup\bp{B_1,\dots,B_k}$-substructure 
of  $(H_{\omega_2},\tau_{\NS_{\omega_1}}^V,B_j:j=1,\dots,k)$ which models $\phi(\vec{A})$, we get that
$\phi(\vec{A})$ holds for $(H_{\omega_2},\tau_{\NS_{\omega_1}}^V,B_1,\dots,B_k)$.

\item[$\phi(\vec{A})\rightarrow \psi(\vec{A})$]
Assume
\[
(H_{\omega_2},\tau_{\NS_{\omega_1}}^V,B_1,\dots,B_k)\models\phi(\vec{A}).
\]
Take any $(\NS_{\omega_1},\UB)$-ec $M\in V$ and any iteration 
$\mathcal{J}=\bp{j_\alpha\beta:\alpha\leq\beta\leq\omega_1}$  of  $M$ witnessing the premises
of the implication in 
$\psi(\vec{A})$, in particular such that:
\begin{itemize}
\item
$\vec{A}=j_{0\omega_1}(\vec{a})\in M_{\omega_1}$ for some $\vec{a}\in M$, 
\item
$\NS_{\omega_1}^{M_{\omega_1}}=\NS_{\omega_1}\cap M_{\omega_1}$,
\item
$M$ is $B_j$-iterable for $j=1,\dots,k$.
\end{itemize}

Such $M$ and $\mathcal{J}$ exists by Thm.~\ref{thm:char(*)}(\ref{thm:char(*)-3}) applied to 
$\bar{E}_{B_1,\dots,B_k}$ and $\vec{A}$.

Let $G$ be $V$-generic for $\Coll(\omega,\delta)$ with $\delta$ inaccessible.
Then in $V[G]$, $V_\delta$ is $\UB^{V[G]}$-correct, by Lemma \ref{lem:UBcorr}.
 
Therefore (since $M$ is $(\NS_{\omega_1},\UB^{V[G]})$-ec also in $V[G]$ by $\maxUB$),
$V[G]$ models that
$j_{0\omega_1^V}$ is a $\Sigma_1$-elementary embedding of
\[
(H_{\omega_2}^{M},\tau_{\NS_{\omega_1}}^{M},B\cap M:B\in\mathsf{UB}_M)
\]
into 
\[
(H_{\omega_2}^V,\tau_{\NS_{\omega_1}}^V,B:B\in\UB_M).
\]
This grants that
\[
(H_{\omega_2}^M,\tau_{\NS_{\omega_1}}^M,B\cap M:B\in\mathsf{UB}_M)\models\phi(\vec{a}),
\]
as was to be shown.
\end{description}
\end{proof}

The Lemma is proved.

\end{proof}

%

\subsubsection{Proof of (\ref{thm:char(*)-modcomp-2})$\to$(\ref{thm:char(*)-modcomp-1})
of Theorem~\ref{Thm:mainthm-1bis}}

\begin{proof}
Assume $\delta$ is supercompact, $P$ is a standard forcing notion to force $\MM^{++}$ of size $\delta$ (such as the one introduced in 
\cite{FORMAGSHE} to prove the consistency of Martin's maximum), and $G$ is $V$-generic for $P$; then
$(*)$-$\UB$ holds in $V[G]$ by Asper\'o and Schindler's recent breakthrough \cite{ASPSCH(*)}.
By Thm. \ref{thm:PI1invomega2}
$V$ and $V[G]$ agree on the $\Pi_1$-fragment of their $\sigma_{\UB^V,\NS_{\omega_1}}$-theory, therefore so do 
$H_{\omega_2}^V$ and $H_{\omega_2}^{V[G]}$ (by \cite[Lemma LEVABS]{VIATAMSTI} 
applied in $V$ and $V[G]$ respectively).

Since $P\in\SSP$
\[
(H_{\omega_2}^V,\tau_{\NS_{\omega_1}}^V,A:A\in \UB^V)\sqsubseteq
(H_{\omega_2}^{V[G]},\tau_{\NS_{\omega_1}}^{V[G]},A^{V[G]}: A\in\UB^V).
\]

Now the model completeness of $T_{\NS_{\omega_1},\UB}$-grants that any of its models (among which $H_{\omega_2}^V$)
is $(T_{\NS_{\omega_1},\UB})_\forall$-ec. This gives that:
\[
(H_{\omega_2}^V,\tau_{\NS_{\omega_1}}^V,\UB^V)\prec_{\Sigma_1}
(H_{\omega_2}^{V[G]},\tau_{\NS_{\omega_1}}^{V[G]},A^{V[G]}: A\in\UB).
\]

Therefore any $\Pi_2$-property for $\sigma_{\UB,\NS_{\omega_1}}$ with parameters in 
$H_{\omega_2}^V$ which holds in
\[
(H_{\omega_2}^{V[G]},\tau_{\NS_{\omega_1}}^{V[G]},A^{V[G]}: A\in\UB)
\]
also holds in $(H_{\omega_2}^V,\tau_{\NS_{\omega_1}}^V,\UB^V)$.

Hence in $H_{\omega_2}^V$ it holds characterization (\ref{thm:char(*)-3}) of $(*)$-$\UB$  given by Thm.~\ref{thm:char(*)} and we are done.
\end{proof}

\subsubsection{Proof of Theorem \ref{thm:char(*)}}

\begin{proof}
Schindler and Asper\'o \cite[Def. 2.1]{SCHASPBMM*++} introduced the following:
\begin{definition} 
Let $\phi(\vec{x})$ be a $\sigma_{\UB,\NS_{\omega_1}}$-formula in free variables $\vec{x}$,
and $\vec{A}\in H_{\omega_2}^V$.
$\phi(\vec{A})$ is \emph{$\UB$-honestly consistent} if for all universally Baire sets 
$U\in \mathsf{UB}^V$,
there is some large enough cardinal $\kappa\in V$ such that 
whenever $G$ is $V$-generic for $\Coll(\omega,\kappa)$, in $V[G]$ there is 
a $\sigma_{\UB,\NS_{\omega_1}}$-structure $\mathcal{M}=(M,\dots)$ such that 
\begin{itemize}
\item
$M$ is transitive and $U$-iterable,
\item
$\mathcal{M}\models \phi(\vec{A})$,
\item
$\NS_{\omega_1}^M\cap V=\NS_{\omega_1}^V$.
\end{itemize}
\end{definition}
They also proved the following Theorem \cite[Thm. 2.7, Thm. 2.8]{SCHASPBMM*++}:
\begin{theorem}
Assume $V$ models $\NS_{\omega_1}$ is precipitous and 
$\maxUB$ holds.

TFAE:
\begin{itemize}
\item $(*)$-$\UB$ holds in $V$.
\item Whenever $\phi(\vec{x})$ is a $\Sigma_1$-formula for $\sigma_{\UB,\NS_{\omega_1}}$ in free variables $\vec{x}$, and $\vec{A}\in H_{\omega_2}^V$,
$\phi(\vec{A})$ is honestly consistent if and only if it is true in $H_{\omega_2}^V$.
\end{itemize}
\end{theorem}

We use Schindler and Asper\'o characterization of $(*)$-$\UB$ to prove the equivalences
of the three items of Thm. \ref{thm:char(*)}
(the proofs of these implications import key ideas
from \cite[Lemma 3.2]{ASPSCH(*)}).

\begin{description}
\item[(\ref{thm:char(*)-1})
implies (\ref{thm:char(*)-2})]
Let $G$ be $V$-generic for $\Coll(\omega,\delta)$.
By Lemma \ref{lem:UBcorr},
$V_\delta$ is $\UB^{V[G]}$-correct in $V[G]$ as witnessed by $\bp{B^{V[G]}:B\in \UB^V}=\UB_V=\bp{B_n^{V[G]}:n\in\omega}$.

\begin{claim}
$V_\delta$ is $(\NS_{\omega_1},\UB^{V[G]})$-ec as witnessed by $\UB_V$.
\end{claim}

\begin{proof}
Let in $V[G]$ $B_V=B_{\UB_V}=\prod_{n\in\omega}B_n^{V[G]}$ be the universally Baire set coding $\UB_V$.

Let $N\leq V_\delta$ in $V[G]$ be $\UB^{V[G]}$-correct with $B_V\in \UB_N$ for some $\UB_N$ witnessing that 
$N$ is $\UB^{V[G]}$-correct.
Then we already observed that $\bp{B^{V[G]}\cap N: B^{V[G]}\in \UB_V}\subseteq \bp{B\cap N: \, B\in \UB_N}$.
Therefore
\[
(H_{\omega_1}^V,\sigma_{\UB_V}^V)=(H_{\omega_1}^V,\sigma_{\UB^V}^V)
\prec (H_{\omega_1}^N,\tau_{\ST}^N, B^{V[G]}\cap N: B\in \UB^V).
\]

Let 
\[
\mathcal{J}=\bp{j_{\alpha,\beta}:\alpha\leq\beta\leq\gamma=(\omega_1)^N}\in N
\] 
be an iteration witnessing
$V_\delta\geq N$ in $V[G]$.

We must show that 
\[
j_{0\gamma}:H_{\omega_2}^V\to H_{\omega_2}^N
\]
is $\Sigma_1$-elementary for $\tau_{\NS_{\omega_1},\UB^V}$ between
\[
(H_{\omega_2}^V,\tau_{\ST}^V,\UB^V,\NS_{\omega_1}^V)
\] 
and 
\[
(H_{\omega_2}^N,\tau_{\ST}^N,B^{V[G]}\cap N: B\in \UB^V,\NS_{\omega_1}^N).
\] 

Let $\phi(a)$ be a $\Sigma_1$-formula for $\tau_{\NS_{\omega_1},\UB^V}$ in parameter 
$a\in H_{\omega_2}^V$ with $B_1,\dots,B_k\in\UB^V$ the universally Baire predicates  occurring in $\phi$
such that 
\[
(N,\tau_{\ST}^N, B^{V[G]}\cap N: B\in \UB^V, \NS_{\omega_1}^N)\models\phi(j_{0\gamma}(a)).
\]
We must show that
\[
(H_{\omega_2}^V,\tau_{\ST}^V, \UB^V, \NS_{\omega_1}^V)\models\phi(a).
\]


Remark that the iteration $\mathcal{J}$ extends to an iteration 
$\bar{\mathcal{J}}=\bp{\bar{j}_{\alpha,\beta}:\alpha\leq\beta\leq\gamma=(\omega_1)^N}$ 
of $V$ exactly as already done in the proof of Lemma \ref{lem:UBcorr}.

Using this observation, let
$\bar{M}=\bar{j}_{0\gamma}(V)$;
then $\NS_{\omega_1}^{\bar{M}}=\NS_{\omega_1}^N\cap \bar{M}$.

Now let $H$ be $V$-generic for $\Coll(\omega,\eta)$ with $G\in V[H]$
for some $\eta>\delta$ inaccessible in $V[G]$.

By $\maxUB$
$N$ is $\UB^{V[H]}$-correct in $V[H]$:
on the one hand 
\[
D_{\UB^{V[H]}}=\Cod[\bar{D}_{\UB^{V[G]}}^{V[H]}],
\]
on the other hand
\[
N\in \Cod[\bar{D}_{\UB^{V[G]}}]\subseteq \Cod[\bar{D}_{\UB^{V[G]}}^{V[H]}].
\]
In particular for any $B\in \UB_V$, $N$ is $B^{V[H]}$-iterable in $V[H]$.

Therefore in $H_{\omega_1}^{V[H]}$ for any $B\in\mathsf{UB}^V$, the statement 
\begin{quote}
\emph{There exists a $\tau_{\NS_{\omega_1}}\cup\bp{B,B_1,\dots,B_k}$-super-structure 
$\bar{N}$ of
$j_{0\gamma}(V_\delta)$
which is
$\bp{B^{V[H]},B_1^{V[H]},\dots,B_k^{V[H]}}$-iterable and which realizes 
$\phi(j_{0\gamma}(a))$}
\end{quote}
holds true as witnessed by $N$.

The following is a key observation:
\begin{subclaim}
For any $s\in (2^{\omega})^{\bar{M}[H]}$ and $B\in \UB^V$
\[
s\in j_{0\gamma}(B)^{\bar{M}[H]}\text{ if and only if } s\in B^{V[H]}\cap \bar{M}[H].
\]
\end{subclaim}
\begin{proof}
For each $B\in \UB^V$ find in $V$ trees $(T_B,S_B)$ which project to complement in $V[H]$ and such that 
$B=p[T_B]$.
Now since $\bar{j}_{0,\gamma}[T_B]\subseteq \bar{j}_{0,\gamma}(T_B)$ and 
$\bar{j}_{0,\gamma}[S_B]\subseteq \bar{j}_{0,\gamma}(S_B)$, we get that 
\begin{itemize}
\item
$(2^{\omega})^{V[H]}=p[[\bar{j}_{0,\gamma}(T_B)]]\cup p[[\bar{j}_{0,\gamma}(S_B)]]$ (since $(2^{\omega})^{V[H]}$ is already covered by $p[[\bar{j}_{0,\gamma}[T_B]]]\cup p[[\bar{j}_{0,\gamma}[S_B]]]$).
\item
$\emptyset=p[[\bar{j}_{0,\gamma}(T_B)]]\cap p[[\bar{j}_{0,\gamma}(S_B)]]$ by elementarity of $\bar{j}_{0,\gamma}$.
\end{itemize}
Hence $B^{V[H]}$ is also the projection of $\bar{j}_{0,\gamma}(T_B)$ 
and the pair $(\bar{j}_{0,\gamma}(T_B), \bar{j}_{0,\gamma}(S_B))$ projects to complement in $V[H]$.

But this pair belongs to $\bar{M}$, and 
(by elementarity of $\bar{j}_{0\gamma}$)
\[
\bar{M}\models(\bar{j}_{0,\gamma}(T_B), \bar{j}_{0,\gamma}(S_B))\text{ projects to complements for
$\Coll(\omega,\bar{j}_{0,\gamma}(\eta))$.}
\]
Since $\eta\leq \bar{j}_{0,\gamma}(\eta)$ we get that 
\[
\bar{M}\models(\bar{j}_{0,\gamma}(T_B), \bar{j}_{0,\gamma}(S_B))\text{ projects to complements for 
$\Coll(\omega,\eta)$.}
\]

Therefore in 
$V[H]$
$s\in j_{0\gamma}(B)^{\bar{M}[H]}$ if and only if $s\in p[[\bar{j}_{0,\gamma}(T_B)]^{V[H]}]\cap M[H]$
if and only if $s\in p[[T_B]^{V[H]}]\cap \bar{M}[H]$ if and only if $s\in B^{V[H]}\cap \bar{M}[H]$.
\end{proof}

This shows that
\[
(\bar{M}[H],\sigma_{\UB^V}^{\bar{M}[H]})\sqsubseteq (V[H],\sigma_{\UB^V}^{V[H]}).
\]

Moreover $H_{\omega_1}^{\bar{M}[H]}$ and  $H_{\omega_1}^{V[H]}$
both realize the theory $T_{\UB^V}$ of $H_{\omega_1}^V$ in this language:
on the one hand 
\[
(H_{\omega_1}^V,\sigma_{\UB^V}^V)\prec (H_{\omega_1}^{\bar{M}},\sigma_{\UB^V}^{\bar{M}})
\prec (H_{\omega_1}^{\bar{M}[H]},\sigma_{\UB^V}^{\bar{M}[H]})
\]
(the leftmost $\prec$ holds since $j_{0,\gamma}:V\to \bar{M}$ is elementary, 
the rightmost $\prec$ holds since $\bar{M}$ models $\maxUB$);
on the other hand
\[
(H_{\omega_1}^V,\sigma_{\UB^V}^V)\prec (H_{\omega_1}^{V[H]},\sigma_{\UB^V}^{V[H]})
\]
(applying  $\maxUB$ in $V$).

Since $T_{\UB^V}$ is model complete, we get that
$H_{\omega_1}^{\bar{M}[H]}$ is an elementary $\sigma_{\UB^V}$-substructure of $H_{\omega_1}^{V[H]}$;
therefore $H_{\omega_1}^{\bar{M}[H]}$ models
\begin{quote}
\emph{There exists a $\tau_{\NS_{\omega_1},B,B_1,\dots,B_k}$-super-structure $\bar{N}$ of
$j_{0\gamma}(V_\delta)$
which is \\
$\bp{\bar{j}_{0\gamma}(B)^{\bar{M}[H]},\bar{j}_{0\gamma}(B_1)^{\bar{M}[H]},\dots,\bar{j}_{0\gamma}(B_k)^{\bar{M}[H]}}$-iterable and which realizes 
$\phi(j_{0\gamma}(a))$.}
\end{quote}

By homogeneity of $\Coll(\omega,\eta)$, in $\bar{M}$ we get that
any condition in $\Coll(\omega,\eta)$ forces:
\begin{quote}
\emph{There exists a $\tau_{\NS_{\omega_1},B,B_1,\dots,B_k}$-super-structure $\bar{N}$ of
$j_{0\gamma}(V_\delta)$
which is \\
$\bp{\bar{j}_{0\gamma}(B)^{\bar{M}[\dot{H}]},\bar{j}_{0\gamma}(B_1)^{\bar{M}[\dot{H}]},\dots,\bar{j}_{0\gamma}(B_k)^{\bar{M}[\dot{H}]}}$-iterable and which realizes 
$\phi(j_{0\gamma}(a))$.}
\end{quote}
By elementarity of $\bar{j}_{0\gamma}$ we get that in $V$ it holds that:
\begin{quote}
There exists an $\eta>\delta$ such that
any condition in $\Coll(\omega,\eta)$ forces:
\begin{quote}
\emph{``There exists a countable super structure $\bar{N}$ of
$V_\delta$ with respect to $\tau_{\NS_{\omega_1},\bp{B,B_1,\dots,B_k}}$
which is $\bp{B^{V[\dot{H}]},B_1^{V[\dot{H}]},\dots,B_k^{V[\dot{H}]}}$-iterable 
and which realizes $\phi(a)$''}
\end{quote}
\end{quote}

This procedure can be repeated for any $B\in\UB^V$, 
showing that $\phi(a)$ is honestly consistent in $V$.

By Schindler and Asper\'o characterization of $(*)$ we obtain that $\phi(a)$ holds in $H_{\omega_2}^V$.
\end{proof}

\item[(\ref{thm:char(*)-2})
implies (\ref{thm:char(*)-3})]
Our assumptions grants that the set 
\[
D_{\UB}=
\bp{M\in H_{\omega_1}^V: M\text{ is $\UB^{V}$-correct}}
\]
is coded by a universally Baire set $\bar{D}_\UB$ in $V$. 
Moreover we also get that whenever $G$ is $V$-generic for 
$\Coll(\omega,\delta)$,
the lift $\bar{D}_{\UB}^{V[G]}$ of $\bar{D}_\UB$ to $V[G]$ codes
\[
D_{\UB^{V[G]}}^{V[G]}=\bp{M\in H_{\omega_1}^{V[G]}: M\text{ is $\UB^{V[G]}$-correct}}.
\]

By (\ref{thm:char(*)-2}) we get that $V_\delta\in D_{\NS_{\omega_1},\UB^{V[G]}}^{V[G]}$.

By Fact \ref{fac:densityUBcorrect}
%
\[
(H_{\omega_1}^V,\tau_{\mathsf{ST}}^V,\UB^V)\models \text{ for all iterable $M$ 
there exists an $\UB$-correct structure 
$\bar{M}\geq M$}.
\]
Again since 
\[
(H_{\omega_1}^V,\tau_{\mathsf{ST}}^V,\UB^V)\prec (H_{\omega_1}^{V[G]},\tau_{\mathsf{ST}}^{V[G]},\UB^V),
\]
and the latter is first order expressible in the predicate $\bar{D}_\UB\in \UB^V$, we get that
\[
(H_{\omega_1}^{V[G]},\tau_{\mathsf{ST}}^{V[G]},\UB^V)\models \text{ for all iterable $M$ 
there exists an $\UB^{V[G]}$-correct structure 
$\bar{M}\geq M$}.
\]
So
let $N\leq V_\delta$ be in $V[G]$ an $\UB^{V[G]}$-correct structure with
$V_\delta\in H_{\omega_1}^N$.

Let $\mathcal{J}=\bp{j_{\alpha\beta}:\,\alpha\leq\beta\leq\gamma=\omega_1^N}\in H_{\omega_2}^N$ be an iteration witnessing $N\leq V_\delta$.

Now for any $A\in \pow{\omega_1}^V$ and $B\in\UB^V$
\[
(H_{\omega_2}^{N},\tau_{\mathsf{ST}}^{N},\NS_{\gamma}^{N},B^{V[G]}\cap N: B\in\UB^V)
\]
models

\begin{quote}
\emph{There exists
an $(\NS_{\omega_1},\UB^{V[G]})$-ec structure $M$ with $B^{V[G]}\cap N\in\UB_M$ and an iteration 
$\bar{\mathcal{J}}=\bp{\bar{j}_{\alpha\beta}:\,\alpha\leq\beta\leq\gamma}$ of $M$ such that
$\bar{j}_{0\gamma}(A)=j_{0\gamma}(A)$}.
\end{quote}
This statement is witnessed exactly by $V_\delta$ in the place of $M$ (since $B=B^{V[G]}\cap V_\delta\in\UB^V$ and 
$\UB^{V[G]}_{V_\delta}=\bp{B^{V[G]}:\, B\in\UB^V}$),
and $\mathcal{J}$ in the place of $\bar{\mathcal{J}}$.

Since $V_\delta$ is $(\NS_{\omega_1},\UB^{V[G]})$-ec in $V[G]$ we get that
$j_{0\gamma}\restriction H_{\omega_2}^V$ is $\Sigma_1$-elementary between $H_{\omega_2}^V$ and $H_{\omega_2}^N$
for $\tau_{\NS_{\omega_1},\UB^V}$.

Hence
\[
(H_{\omega_2}^{V},\tau_{\mathsf{ST}}^{V},\NS_{\gamma}^{V},\UB^V)
\]
models 
\begin{quote}
\emph{There exists 
an $(\NS_{\omega_1}^V,\UB^{V})$-ec structure $M$ with $B\in\UB_M$ and an iteration 
$\bar{\mathcal{J}}=\bp{\bar{j}_{\alpha\beta}:\,\alpha\leq\beta\leq(\omega_1)^V}$ of $M$ such that
$\bar{j}_{0\omega_1}(a)=A$ and 
$\NS_{\omega_1}^{\bar{j}_{0\omega_1}(M)}=\NS_{\omega_1}^V\cap \bar{j}_{0\omega_1}(M)$}.
\end{quote}

\item[(\ref{thm:char(*)-3})
implies (\ref{thm:char(*)-1})]
We use again Schindler and Asper\'o characterization of $(*)$.

Assume $\phi(A)$ is honestly consistent for some $\Sigma_1$-property $\phi(x)$ in the language 
$\sigma_{\UB,\NS_{\omega_1}}$
and $A\in\pow{\omega_1}^V$.
Let $B_1,\dots,B_k$ be the universally Baire predicates in $\UB$ mentioned in $\phi(x)$.

By (\ref{thm:char(*)-3}) there is in 
$V$ an $(\NS_{\omega_1},\UB)$-ec $M$ with $B_1,\dots,B_k\in \UB_M$ 
and $a\in \pow{\omega_1}^M$,
and an iteration $\mathcal{J}=\bp{j_{\alpha\beta}:\,\alpha\leq\beta\leq\omega_1}$ of $M$ such that
$j_{0\omega_1}(a)=A$ and $\NS_{\omega_1}^{j_{0\omega_1}(M)}=\NS_{\omega_1}^V\cap j_{0\omega_1}(M)$.

Let $G$ be $V$-generic for $\Coll(\omega,\delta)$. 
Find $N\in V[G]$ such that $N\models \phi(A)$, $N$ is 
$B_1^{V[G]},\dots,B_k^{V[G]}$-iterable in 
$V[G]$ and 
$\NS_{\omega_1}^{N}\cap V=\NS_{\omega_1}^V$ (this $N$ exists by the honest consistency of $\phi(x)$).

Notice that $\mathcal{J}\in V_\delta\subseteq N$ witnesses that $M\geq N$ as well.

Let $\bar{N}\leq N$ in $V[G]$ be a $\UB^{V[G]}$-correct structure with $B_{\UB_V}\in \UB_{\bar{N}}$
($\bar{N}$ exists by Fact \ref{fac:densityUBcorrect} applied in $V[G]$ to $N$ and $B_{\UB_V}$). 
Let $\mathcal{K}=\bp{k_{\alpha\beta}:\alpha\leq\beta\leq\bar{\gamma}=\omega_1^{\bar{N}}}\in\bar{N}$ be an
iteration witnessing that $\bar{N}\leq N$.

Remark that $H_{\omega_2}^{\bar{N}}\models \phi(k_{0\bar{\gamma}}(A))$, since $\Sigma_1$-properties are
upward absolute and $k_{0\bar{\gamma}}(N)$ is a 
$\tau_{\NS_{\omega_1}}\cup\bp{B_1,\dots,B_k}$-substructure of $H_{\omega_2}^{\bar{N}}$.

Also $\bp{B^{V[G]}:B\in\UB_V}\subseteq \UB_{\bar{N}}$ entail that
$B_{\UB_M}^{V[G]}\in \UB_{\bar{N}}$.

Letting 
\[
\bar{\mathcal{J}}=\bp{\bar{j}_{\alpha\beta}:\alpha\leq\beta\leq\bar{\gamma}}=k_{0\bar{\gamma}}(\mathcal{J}),
\]
we get that $\bar{j}_{0\bar{\gamma}}(a)=k_{0\gamma}(j_{0\bar{\gamma}}(a))=k_{0\gamma}(A)$, and
$\bar{\mathcal{J}}$ is such that $B_j^{V[G]}\in \UB_{\bar{N}}$ for all $j=1,\dots,k$ since
$B_{\UB_M}^{V[G]}$ in $\UB_{\bar{N}}$. 

Since $M$ is $(\NS_{\omega_1},\UB^{V[G]})$-ec in $V[G]$ by $\maxUB$, we get that
$\bar{j}_{0\bar{\gamma}}$ defines a $\Sigma_1$-elementary embedding of
\[
(H_{\omega_2}^M,\sigma_{\UB_M,\NS_{\omega_1}}^M)
\]
into 
\[
(H_{\omega_2}^{\bar{N}},\sigma_{\UB_M,\NS_{\omega_1}}^{\bar{N}}).
\]

%
%
Hence 
\[
(H_{\omega_2}^M,\sigma_{\UB_M,\NS_{\omega_1}}^M)\models\phi(a).
\] 

This gives that
\[
(H_{\omega_2}^{M_{\omega_1}},\sigma_{\UB_M,\NS_{\omega_1}}^{M_{\omega_1}})\models\phi(A)
\] 
(since $j_{0\omega_1}(a)=A$), 
and therefore that
\[
(H_{\omega_2}^V,\sigma_{\UB_M,\NS_{\omega_1}}^V)\models\phi(A),
\]
since $M_{\omega_1}$ is a substructure of $H_{\omega_2}^V$ for $\sigma_{\UB_M,\NS_{\omega_1}}$.

\end{description}
\end{proof}


\begin{question}
Is the use of $\maxUB$ really necessary? It is not at all clear whether the chain 
of equivalences for $(*)$-$\UB$ could be proved replacing it with the usual Woodin's axiom $(*)$ 
as formulated in
\cite[Def. 7.5]{HSTLARSON}; in all cases where the argument appealed to $\maxUB$ one should find a
different strategy to reach the desired conclusion. 
\end{question}

\bibliographystyle{plain}
	\bibliography{Biblio}

\end{document}